\documentclass[12pt,reqno]{amsart}
\usepackage{amsfonts,amssymb,amsmath}
\usepackage{enumitem}
\usepackage{ytableau}
\usepackage{datetime}
\usepackage{tikz}
\usepackage{xcolor}
\newcommand\blue[1]{\textcolor{blue}{#1}}

\newcommand\magenta[1]{\textcolor{magenta}{#1}}

\newcommand\violet[1]{\textcolor{violet}{#1}}
\definecolor{fgreen}{RGB}{44,144, 14}

\renewenvironment{proof}{{\bfseries Proof.}}{\qed}

\numberwithin{equation}{section} 
\newtheorem{theorem}{Theorem}[section] 
\newtheorem{proposition}[theorem]{Proposition} 
\newtheorem{corollary}[theorem]{Corollary} 
\newtheorem{lemma}[theorem]{Lemma}

\theoremstyle{definition}
\newtheorem{definition}[theorem]{Definition} 
\newtheorem{remark}[theorem]{Remark} 
\newtheorem{example}[theorem]{Example}
\usepackage[colorlinks=true,  citecolor=cyan,urlcolor=blue,  
linkcolor=blue]{hyperref}

\def\D{\mathbb {D}}
\def\R{\mathbb {R}}
\def\C{\mathbb {C}}
\def\N{\mathbb {N}}
\def\H{\mathbb {H}}

\def\O{\mathbb {O}}
\def\E{\mathbb {E}}
\def\F{\mathbb {F}}

\def\ib{\mathbf {i}}
\def\jb{\mathbf {j}}
\def\kb{\mathbf {k}}
\def\d{\mathbf{ d}}

\def\PC{\mathcal {P}}

\def\OC{\mathcal {O}}
\def\BC{\mathcal {B}}

\def\g{\mathfrak {g}}

\def\p{\mathfrak {p}}

\def\s{\mathfrak {s}}

\def\u{\mathfrak {u}}
\def\o{\mathfrak {o}}
\def\l{\mathfrak {l}}
\def\z{\mathfrak {z}}

\def\th{\theta}

\def\<>{\langle \cdot\, , \cdot \rangle}
\def\lto{\longrightarrow}

\newcommand{\defref}[1]{Definition~\ref{#1}}

\begin{document} 
 \title[Reality of Unipotent Elements]{Reality of Unipotent Elements in Classical Lie Groups}
 \author[K. Gongopadhyay  \and C. Maity]{Krishnendu Gongopadhyay \and Chandan Maity
 }
\address{Indian Institute of Science Education and Research (IISER) Mohali,
 Knowledge City,  Sector 81, S.A.S. Nagar 140306, Punjab, India}
\email{krishnendug@gmail.com, krishnendu@iisermohali.ac.in}
\address{Indian Institute of Science Education and Research (IISER) Mohali,
 Knowledge City,  Sector 81, S.A.S. Nagar 140306, Punjab, India}
\email{maity.chandan1@gmail.com, cmaity@iisermohali.ac.in}

\begin{abstract}
The aim of this paper is to give a classification of real and strongly real unipotent elements in a classical simple Lie group. To do this, we will introduce an infinitesimal version of the notion of  classical reality in a Lie group. This notion has been applied to classify real and strongly real unipotent elements in a classical simple Lie group.   
\end{abstract}
\maketitle
\tableofcontents

\section{Introduction} 
\subsection{Reality and Strong Reality  in Groups}  
Recall that an element $g$ in a group $G$ is called \emph{real} or \emph{reversible} if it is conjugate to $g^{-1}$ in $G$ that is there is a $h$ in $G$ such that $g^{-1}=hgh^{-1}$. The element $g$ is called \emph{strongly real} or \emph{strongly reversible} if $g$ is a product of two involutions in $G$, equivalently if there is an involution $h\in G$ so that $g^{-1}=hgh^{-1}$, then $g$ is strongly real. Thus a strongly real element in $G$ is real, but a real element is not necessarily strongly real.

It has been a problem of wide interest to investigate the real and strongly real elements in groups, see \cite{FS} for an exposition of this theme. 
However, in the literature, most of the investigations address the converse question, that is when a real element is strongly real, e.g.,  in \cite{ST} Singh and Thakur proved the equivalence of reality and strong reality in a connected adjoint semisimple algebraic group over a perfect field provided $-1$ is an element in the Weyl group. It follows from \cite{ST} that in most classical groups over a field, a semisimple element is real if and only if it is strongly real. A  classical family where this does not hold is the symplectic group. 

A complete classification of real and strongly real elements is known mostly for a few families of finite groups, e.g., \cite{vg}, \cite{g81}.  
Investigations of the reality of unipotent elements are also focused around finite groups. Tiep and Zalesski investigated this question in a finite group of Lie types in \cite{tz2}. Recently close relationship between strong reality in finite simple groups and their totally orthogonal representations has been established in \cite{vi}.

Given an infinite group, a complete classification of real and strongly real elements is not known for most of the cases except for very few families, see \cite{FS}. The only classical Lie groups where concrete classifications are known are the general linear groups, compact groups and real rank one classical groups, e.g.  \cite{FS, GP, BG}. 

The aim of this paper is to give a complete classification of the real and strongly real unipotent elements in a classical simple Lie group. In the literature, the investigations of the reality problem have been done mostly from the viewpoint of linear algebra and the theory of algebraic groups. In some geometric cases, e.g. real rank one, local geometry of the transformation has also seen a role. Even though there is a complete classification for simple Lie groups, we have not seen any work where general structure theory of Lie groups has been used to tackle the problem. This work is an attempt to fill this gap in the literature and initiate an investigation of reality using the adjoint representations. For this, we introduce an infinitesimal notion of reality, viz. {\rm Ad}$_G$-real given below. This infinitesimal reality not only helps us to classify the unipotent classes, but it may be a problem of independent interest to investigate these classes in their own rights.

\subsection{Adjoint Orbits and Reality}\label{sec-adjoint-orbits}
 Let $G$ be a Lie group with Lie algebra $\g$. Consider the natural {\rm Ad}$(G)$-representation of $ G $ on its Lie algebra $ \g $
 $$  
 {\rm Ad} \colon G \lto {\rm GL}(\g)\,.
 $$
 For $X\in \g$, {\it the adjoint orbit}  of $X$ in $\g$ is defined as $\OC_X:= \{{\rm Ad}(g)X\mid g \in G \}$.
 If $X\in \g$ is nilpotent, then $\OC_X$ is called {\it nilpotent orbit}. Many important and deep results about such orbits are available in the literature, e.g. \cite{CoMc}, \cite{Mc}. We relate this theory with the above problem of reality in Lie groups. To connect the two themes, we define the following. 
 \begin{definition}\label{def1}
 	An element $X\in \g$ is called {\it {\rm Ad}$_G$-real } if $-X \in \OC_X$. An  {\rm Ad}$_G$-real  element is called {\it strongly {\rm Ad}$_G $-real}  if $-X = {\rm Ad}(\tau) X $ for some $\tau\in G$ so that $\tau ^2 = {\rm Id}$.
 \end{definition}  
 We will show that the classical reality of the unipotent elements in a Lie group and Ad$ _G$-reality of the nilpotent elements are equivalent, but this is not true in general.
 To see this let $ {\rm I}_{1,1} = {\rm diag }(1,-1) = \exp Y \in {\rm GL}_2(\C)$, where   $Y =  {\rm diag }( \sqrt{-1} 2\pi, \sqrt{-1}\pi)   $. Then $ {\rm I}_{1,1} $ is strongly real but $Y$ is  not Ad$_{{\rm GL}_2(\C)}$-real in $ \g\l_2(\C) $.
 
 In this paper, by a real element in the Lie algebra $ \g $, we mean the {\rm Ad}$ _G$-real element whenever underlined Lie group $ G $ is understood from the context. Similarly,  by the real element in a Lie group $ G $, we mean reality in the classical sense.
 We shall classify the {\rm Ad}$ _G$-real nilpotent classes in a simple Lie algebra, and that leads us to the classification of real unipotent classes in the corresponding  Lie group using the exponential map; see Section  \ref{sec-Ad-reality-vs-reality}.

\subsection{Summary of our results} First, we shall consider the complex simple Lie groups, i.e., the corresponding Lie algebra $\g$ is a simple Lie algebra over $\C$.
We will provide an elementary proof of the fact  that every unipotent element in a complex semi-simple Lie group $G$  is real without using Jacobson-Morozov Theorem, see  Lemma \ref{thm-real-nilp-g/C}. 
However, not every unipotent is strongly real in general, and we classify precisely which unipotent elements are strongly real. 
 The  non-trivial and most technical part of the paper is the classification of strongly real unipotent elements. 
Next, we will consider the real simple Lie groups and classify real and strongly real unipotent elements in those groups. As mentioned above, we have not seen any place in the literature where the problem of reality has been addressed using Lie theoretic tools. Rather it has been using tools from geometric algebra. So, to follow that tradition, we divide the real and complex Lie groups as per their geometric algebra classification and summarize our results in the context of each class.  

In the following let $\D=\R$, $\C$ or $\H$ unless stated otherwise.  Let $ {\rm GL}_n(\D) $ be the general linear group over $ \D $. Other than $ {\rm GL}_n(\D) $, we have worked with the classical Lie groups and Lie algebras as listed in Section \ref{sec-classical-Lie-gps-Lie-algs}. In the following, we summarize our results concerning these groups.

\subsubsection{The general linear group}
The real and strongly real elements in $\mathrm{ GL}_n( \C)$ and  $\mathrm{ GL}_n( \R)$  are well-understood in the literature, e.g. \cite{W}, \cite{FS}.  However, it is not so well-known in the literature for $\mathrm{ GL}_n( \H)$. We have given a uniform approach to prove that every unipotent element in $\mathrm{GL}_n(\D)$ is strongly real, see Proposition \ref{Propo-glnD}. 

\subsubsection{The special linear group}  The classification of real and strongly real elements in $\mathrm{SL}_n(\D)$ is more tricky than the general linear case. It is easy to see that the unipotent elements in $\mathrm{SL}_2(\R)$ are not even real. General investigation of involution length in $\mathrm{SL}_n( \F)$, for  a field $\F$, has been done by many authors, e.g. \cite{KN2}, \cite{KN3}. However, in view of the classification problem of real elements, the best result that we know states that for  $ n\not\equiv 2\pmod4 $  an element in $\mathrm{SL}_n(\F)$ is real if and only if it is strongly real, where $\F$ is a field of characteristic not $2$, see \cite[Theorem 3.1.1]{ST}. In this work, we ask for a precise count of the strongly real unipotent elements. We have given a complete characterization of real and strongly real unipotent elements in $\mathrm{SL}_n(\D)$. For $\D=\R$ or $\C$, not every unipotent element is strongly real, and we classify precisely when a unipotent is strongly real, see Theorem \ref{thm-sl-n-c-nilpot} and Theorem \ref{thm-sl-n-R-nilpot}. For $\mathrm{SL}_n(\H)$,  we show that every unipotent element is strongly real, see Theorem \ref{thm-sl-n-H}. 

\subsubsection{The symplectic group}Let $\F=\R$ or $\C$.  Let  ${\rm{ Sp}} (n, \F)$ denote the symplectic group over  $\F$.  It is known that not every element is strongly real in ${\rm Sp}(n, \F)$. It follows from the literature that every element in ${\rm{ Sp}} (n, \F)$ is a product of two skew-involutions, i.e. every element in ${\rm{PSp}} (n, \F)$ is strongly real, see \cite{ST}, \cite{W}.   In a recent work, Ellers and Villa \cite{EV} have proved decomposition of a symplectic map into a product of six involutions and have classified certain strongly real elements under some strong hypothesis, see \cite[Corollary 9]{EV}. In \cite{deC}, de la Cruz has given a characterization of strongly real elements in ${\rm{ Sp}} (n, \C)$. It is proved in  \cite{deC2}  that every element in ${\rm{Sp}}(2, \R)$ is a product of four involutions. However, strongly real classes are not very well-understood in ${\rm{ Sp}} (n, \R)$. In this work, we have given a very first classification of unipotent strongly real elements in ${\rm{ Sp}}(n, \R)$, see Theorem \ref{thm-sp-n-R-nilpot}.

\subsubsection{The orthogonal and unitary groups} 
Product of involutions in orthogonal groups (over general fields) have been characterized in the literature, e.g. \cite{FZ}, \cite{KN1}, \cite{KN2}. However, we do not know any concrete classification of real elements in ${\rm O}(p,q)$ or ${\rm{SO}}(p,q)$ other than when $p=1, q>0$, see \cite{FS}, \cite{G}, \cite{BM}. The same can be said about unitary groups  ${\rm U}(p,q)$. In \cite{DM2}, it has been proved that every element in ${\rm{U}}(p,q)$ is a product of finitely many reflections, and an estimate of the reflection length was given. However, concrete classification  of real and strongly real elements are known only for   ${\rm U}(n,1)$ and ${\rm{SU}}(n,1)$, see \cite{GP}. The classification of real and strongly real elements in $\mathrm{Sp}(n,1)$ has been obtained very recently in \cite{BG}. Other than the rank one groups, not much is known about a precise classification of the real elements in  ${\rm{SU}}(p,q)$ or ${\rm{Sp}}(p,q)$.

In this paper, we classify real and strongly real unipotent elements  in   ${\rm{SU}}(p,q)$ and ${\rm{Sp}}(p,q)$. We prove that every unipotent element in  ${\rm{Sp}}(p,q)$ is real, see Corollary \ref{cor-sp-pq-real}.   As we shall see that is not the case in the groups ${\rm{SU}}(p,q)$. Though reality and strong reality are equivalent in this group, not every unipotent element is real. We ask for a precise list of real or strongly real unipotent elements in $\mathrm{SU}(p,q)$, ${\rm{Sp}}(p,q)$  and classify such elements, see Theorem \ref{thm-supq-real-streal} and Theorem \ref{thm-sp-pq-st-real}.   
We have proved that every unipotent element in $\mathrm{SO}(p,q)$ is strongly real; see Theorem \ref{propo-so-pq-unilpotent}.

\subsubsection{The group $\mathrm{SO^{\ast}}(2n)$} The group $\mathrm{SO^{\ast}}(2n)$ is a real form of the complex orthogonal group $\mathrm{SO}(2n, \C)$, and is a subgroup of $\mathrm{SL}_n(\H)$. We have not seen any place in the literature where reality in ${\rm{SO^{\ast}}}(2n)$ has been addressed. In this paper, we give a very first account of the classification of real and strongly real unipotent elements in this group, see Theorem \ref{thm-so*2n-nilpot}. We hope to address the semisimple case in a later investigation. 

\medskip

A natural problem following \defref{def1} is classification of the \text{Ad}$_G$-real elements in $\g$. For completeness, we have completely classified {\rm Ad}$ _G $-reality of the semisimple and nilpotent  elements in complex simple Lie algebras.   The situation for real Lie algebras seems more tricky. We hope to address this in subsequent work.

\medskip

Another important fact is that our results describe the reversing symmetry group or extended centralizer of real unipotent elements, which is defined as follows. 
For a group $ G $, the {\it  centralizer and reverser } of an element $g$ in $G$ are respectively defined as 
$${Z}_G(g): = \{ s \in G \mid sgs^{-1} = g \}, \hbox{ and \ } {R}_G(g) := \{ r \in G \mid rgr^{-1} = g^{-1} \}.$$
Note that the set ${R}_G(g)$  is a right coset of the centralizer ${Z}_G(g)$ of $g$. Thus the { reversing symmetry group} or { extended centralizer} ${E}_G(g) := {Z}_G(g) \cup {R}_G(g)$ is a subgroup of $G$ in which ${Z}_G(g)$ has index $1$ or $2$,  see {\cite[Proposition 2.8]{FS}}.   
The group ${E}_G(g)$ is an extension of ${Z}_G(g)$ of degree at most two. Finding the reversing symmetric group ${E}_G(g)$, it is enough to construct one reversing element which is not in the centralizer.  
For an elaborate discussion on reversing symmetry groups; see {\cite[\S 2.1.4]{FS}}, \cite{BR}. We have explicitly constructed an element in $ R_G(g) $ for each real unipotent element. Accordingly, this classifies the reverser of real unipotent elements in simple Lie groups.

\subsection{Key tools}
The main strategy of our proof depends on the theory of  nilpotent orbits. More specifically,   we will use certain ordered bases which were constructed to describe centralizers of the nilpotent elements in classical simple Lie algebras as in  \cite{SS}, \cite{BCM}.   For base field $\R$ or $\C$,   a specific basis was obtained in \cite{SS} to describe the reductive parts of the centralizers of unipotent elements.   In \cite{BCM}, construction of such basis was extended to $ \H $, and was obtained in a unified way to $\D$ using the Jacobson-Morozov theorem and the structure of the related $ \s\l_2(\R) $-representations. These special bases are used to construct conjugating elements. Another key tool to classify the Ad$_G $-real elements is the parametrization of nilpotent orbits in simple Lie algebras involving the signed Young diagrams. 
A detail analysis involving the structure of the centralizer of the associated nilpotent elements are  used in the classification of the strongly Ad$_G  $-real elements.

\subsection{Structure of the paper} 
The paper is organized as follows. In Section \ref{sec-adjoint-orbits}, the notion of  {\rm Ad}$_G$-real in the Lie algebra is introduced.   Relationship between {\rm Ad}$_G$-real and classical reality is established in Section \ref{sec-Ad-reality-vs-reality}. In Section \ref{sec-notation}, we fix some notation and recall some background.   
In Section \ref{sec-reality-nilpotent-elts-complex-g} {\rm Ad}$_G$-real and strongly {\rm Ad}$_G$-real nilpotent elements in complex simple classical Lie algebras are classified. Finally, in Section \ref{sec-reality-nilpotents--real-g}, we have classified  {\rm Ad}$_G$-real and strongly {\rm Ad}$_G$-real nilpotent elements in simple classical Lie algebras over $\R$.

\section{${\mathrm{Ad}}_G$-reality and Classical Reality}\label{sec-Ad-reality-vs-reality}
In this section we will introduce an infinitesimal  version of the notion of classical reality in a Lie group. 
First, we will show that  {\rm Ad}$_G$-real in the Lie algebra implies the classical reality of the corresponding element in the Lie group.

\begin{lemma}\label{lem-Ad-real-to-classical}
Let $ G $ be a Lie group with Lie algebra $ \g $. Let $ X\in \g$ be an {\rm Ad}$_G$-real  (resp. strongly {\rm Ad}$_G$-real)  element, then $ \exp X$ is real	 (resp. strongly real) in $ G $.
\end{lemma}

 There is  a $ G $-equivariant bijection between the set of nilpotent elements in $ \g $ and the set of the unipotent elements in $ G $ via the exponential map.

\begin{corollary}\label{cor-Ad-real-unipotent-elt}
Let $ G $ be a semi-simple linear Lie group, and $ u\in G $ be a unipotent element so that $ u=\exp X $. Then $ u $ is real (resp. strongly real) if and only if $ X $ is ${\rm Ad}_G $-real (resp. strongly  ${\rm Ad}_G $-real ) in $ \g$. 
\end{corollary}

\begin{remark}
For a connected simply connected  nilpotent Lie group $ N $ with Lie algebra $ \mathfrak{n} $, the exponential map $ \exp \colon \mathfrak{n} \to N $ is a diffeomorphism; see \cite[Theorem 1.127]{K}. Thus	
The classical reality of a connected simply connected  nilpotent Lie group $ N $  is  determined by the {\rm Ad}$ _N$-reality in $\mathfrak{n} $.	  \qed
\end{remark}

It is not true in general for a Lie group $G$ that $g =\exp X \in G$ is real (resp. strongly real)  if and only if  $ X $ is {\rm Ad}$ _G $-real (resp. strongly {\rm Ad}$ _G $-real).  We give examples to illustrate this  situation.

\begin{example}\label{rmk-reality-vs-Adreal}
{\bf (i)~} Let $g= - {\rm Id} \in {\rm SL}_2(\C)$ and $ X= {\rm diag}(	{\sqrt{ -1}\,\pi}, -	{\sqrt{ -1}\,\pi} ) $. Then $ g= e^X $. Note that $ g $ is strongly real in ${\rm SL}_2(\C)$. But $ X $ is not a strongly {\rm Ad}$ _G $-real element in $ \s\l_2(\C) $.
	
{\bf (ii)~} Next we will consider the compact Lie group $ {\rm SO} _2$.  The element $ \sigma={\rm -I}_2=\exp \begin{pmatrix}
	0 & \pi\\
	-\pi&0
	\end{pmatrix} \in {\rm SO}_2$ is real. But $\begin{pmatrix}
	0 & \pi\\
	-\pi&0
	\end{pmatrix}$ is not {\rm Ad}$ _G $-real in the Lie algebra $\s\o_2$.  \qed
\end{example}

\section{Notation and background}\label{sec-notation}
In this section, we fix some notation and recall some known results which will be used in this paper. We will follow the notation of \cite{BCM}.  

Once and for all fix a square root of $-1$ and call it $\sqrt{-1}$.  The Lie groups will be denoted by the capital letters,  while the Lie algebra of a Lie group will be denoted by the corresponding lower case 
German letter.  Sometimes, for notational convenience, the Lie algebra of a Lie group $G$ is also denoted by  ${\rm Lie} (G)$.  The connected component of a Lie group $G$  containing the identity element  is denoted by $G^{0}$.  For a subset $S$ of $\g$, the   subgroup of $G$ that fixes $S$ point wise under the adjoint action is called the {\it centralizer}  of $S$ in $G$; the centralizer of $S$ in $G$ is denoted by $Z_{G} (S)$.  Similarly, for a Lie subalgebra $\g$ and a subset $S \,\subset\, \g$, by $\z_\g (S)$ we will denote the subalgebra of $\g$ consisting of all the elements that commute with every element of $S$.

\subsection{Associated Lie groups and Lie algebras}\label{sec-Associated-Lie-gp}
Let $\D=\R$, $\C$ or $\H$. Let $V$ be a right vector space of finite dimension over $\D$. 
Let ${\rm End}_\D (V)$ be the real algebra of {\it right $\D$-linear maps} from $V$ onto $V$. 
After choosing a basis for $V$,  the elements of ${\rm End}_\D (V)$ can be represented by matrices over $\D$. 
For a $\D$-linear map $T \,\in\, {\rm End}_\D (V)$ and an ordered $\D$-basis $\BC$ of $V$,
the {\it matrix of $T$ with respect to $\BC$} is denoted by $[T]_{\BC}$.
Let  ${\rm GL(V)}$ denote the {\it group of invertible} right  $ \D $-linear maps from ${\rm End}_\D (V)$.

When $\D=\R$ or $\C$, let
$
{\rm tr} \,:\, {\rm End}_\D (V) \,\longrightarrow\, \D\, \ \ \text{ and }\ \
{\rm det} \,:\, {\rm End}_\D V \,\longrightarrow\, \D
$,
 respectively be the usual {\it trace} and {\it determinant} maps.  Define 
$${\rm SL}(V)\,:=\, \{ g \,\in \,{\rm GL}(V) \, \mid \, \det(g)\,= \,1\} \ \ \text{ and } 
\ \ \s\l (V ) \,:=\, \{ T \,\in\, {\rm End}_\D (V) \, \mid \, \text{tr}(T)\,=\, 0 \}\, .$$

For $\D = \H$, recall that $A={\rm End}_\D(V)$ is a central simple $\R$-algebra. Let
$\text{Nrd}_A\,:\, A \,\longrightarrow\, \R$ be the {\it reduced norm} on $A$, and
let $\text{Trd}_A \,:\, A \,\longrightarrow\, \R$ be the {\it reduced trace} on $A$. Define
$$
{\rm SL}(V):= \{ g \in {\rm GL}(V)  \mid  \text {Nrd}_{ {\rm End}_\D V} (g)= 1\}  \text{ and } \s\l (V ):= \{ T\in {\rm End}_\D(V) \mid 
\text{Trd}_{ {\rm End}_\D (V)} (T)= 0\}.
$$

Let $\D=\R$, $\C$ or $\H$, as above.
Let $\sigma$ be either the identity map $\text{Id}$ or an {\it involution} of $\D$, that is,  
$\sigma$ is a real linear anti-automorphism of $\D$ with $\sigma^2 \,=\, \text{Id}$.  Let $\epsilon \,=\, \pm 1$.
Following \cite[\S~23.8, p. 264]{Bo}, we call a map 
$\,\langle \cdot,\,
\cdot \rangle \,\colon\, V \times V \,\longrightarrow\, \D\,$ a
$\epsilon$-$\sigma$ {\it Hermitian form} if 
\begin{enumerate}[label=(\roman*)]
	\item $\langle u+v, w \rangle =\langle u, w  \rangle + \langle v, w \rangle$, 	
	\item $ \langle v,\, u \rangle \,= \, \epsilon \sigma( \langle u, v \rangle)$, and
	
	\item $ \langle v \alpha,\, u \rangle \,=\, \sigma (\alpha) \langle v,\, u \rangle$ for all
	$v,\,u ,\,w\,\in\, V$ and for all $\alpha \,\in\, \D$.
\end{enumerate}

Recall that A $\epsilon$-$\sigma$ Hermitian form $ \langle \cdot, \, \cdot \rangle $ is {\it
	non-degenerate} if for every non-zero $u$ in $V$, there is a non-zero $v$ such that $ \langle v,\, u \rangle \,\neq \,0 $.  We define 
\begin{align}\label{defn-U-gp-4-epsilon-delta-form}
{\rm U} (V,\, \langle \cdot,\, \cdot \rangle ) \,:=\,
\{g \,\in\, {\rm GL}(V) \, \mid \, \langle gv ,\, gu \rangle \,= \, \langle v ,
\, u \rangle ~~ \forall~~ v,u \,\in \,V \}
\end{align}
and
$$
\u (V,\, \langle \cdot, \,\cdot \rangle ) \,:=\, \{T\,\in\, {\rm End}_\D(V) \, \mid \, 
\langle Tv ,\, u \rangle + \langle v ,\, Tu \rangle \,= \,0 ~~\forall~~ v,u \,\in \,V \}\, .
$$ 
We next define $${\rm SU} (V, \langle \cdot, \cdot \rangle )\,: =\,
{\rm U} (V, \langle \cdot, \cdot \rangle ) \cap {\rm SL}(V)\ \ \text{ and } \ \
\s\u (V, \langle \cdot, \cdot \rangle ) \,:=\,
\u (V, \langle \cdot, \cdot \rangle ) \cap \s\l (V)\, .$$ 
It is well-known that  $\s\u (V, \,\langle \cdot,\,\cdot \rangle )$ is
a simple Lie algebra (cf. \cite[Chapter I, Section 8]{K}). 

\subsection{Classical simple Lie groups and Lie algebras}\label{sec-classical-Lie-gps-Lie-algs}
For a suitable choice of $ V $ and $ \<> $, we have the following classical Lie groups and Lie algebras which will be used here.  We define the {\it usual conjugations} $\sigma_c$ on $\C$ by
$\sigma_c (x_1 + \sqrt{-1}x_2 ) \,=\, x_1-\sqrt{-1}x_2$, and on 
$\H$ by $\sigma_c (x_1 + \ib x_2 + \jb x_3 + \kb x_4 ) \,=\, x_1 - \ib x_2 - \jb x_3 - \kb x_4$, $x_i \in \R$ for $i= 1, \dots , 4$. For $P\,=\, (p_{ij}) \,\in\, {\rm M}_{r\times s}(\D)$,  $P^t$ denotes the {\it transpose} of $P$. If $\D \,=\, \C$ or $\H$, then define $\overline{P} \,:=\, (\sigma_c(p_{ij}))$. Let
\begin{equation}\label{defn-I-pq-J-n}
{\rm I}_{p,q} \,:=\, \begin{pmatrix}
{\rm I}_p  \\
& -{\rm I}_q
\end{pmatrix}\, , ~~~ \
{\rm J}_n \,:=\, \begin{pmatrix}
& -{\rm I}_n  \\
{\rm I}_n & 
\end{pmatrix}\,.
\end{equation}

We will work with the following classical Lie groups and Lie algebras:
\begin{align*}
{\rm SL}_n (\C)&:=\, \{g \,\in\, {\rm GL}_n (\C) \,\mid\, \det (g) \,=\,1  \}, 
&{\s\l}_n (\C):=\{z \,\in\, {\rm M}_n (\C) \,\mid\, \text{tr} (z) \,=\,0 \};
\\
{\rm SO} (n,\C)&:=\, \{g \,\in\, {\rm SL}_{n}(\C)\,\mid\, g^t g \,=\, {\rm I}_{n} \},
&{\s\o} (n,\C):= \{z \,\in\, \s\l_{n}(\C) \,\mid\, z^t {\rm I}_{n} + {\rm I}_{n} z \,=\,0 \};\\
{\rm Sp} (n,\C)&:=\, \{g \,\in\, {\rm SL}_{2n}(\C) \,\mid\, g^t {\rm J}_{n} g \,=\,{\rm J}_{n} \}, 
&{\s\p} (n,\C):= \{z \,\in\, \s\l_{2n}(\C) \,\mid\, z^t {\rm J}_{n} + {\rm J}_{n} z \,=\,0 \};
\end{align*}
\begin{align*}
&{\rm SL}_n (\R):=\, \{g \,\in\, {\rm GL}_n (\R) \,\mid\, \det (g) \,=\,1  \}, \qquad \qquad
{\s\l}_n (\R)\,:=\, \{z \,\in\, {\rm M}_n (\R) \,\mid\, \text{tr} (z) \,=\,0 \};
\\
&{\rm SL}_n (\H):=\, \{g \,\in\, {\rm GL}_n (\H) \,\mid\, \text{Nrd}_{{\rm M}_n (\H)} (g) \,=\,1 \}, \,  
{\s\l}_n (\H):= \{z \,\in\, {\rm M}_n (\H) \,\mid\, \text{Trd}_{{\rm M}_n (\H)} (z) \,=\,0 \};
\\
&{\rm SU} (p,q):=\{g \in\, {\rm SL}_{p+q}(\C) \,\mid\, \overline{g}^t{\rm I}_{p,q} g \,= \,{\rm I}_{p,q} \} , \ \
{\s\u} (p,q):= \{z \,\in\, \s\l_{p+q}(\C)\,\mid\, \overline{z}^t{\rm I}_{p,q} + {\rm I}_{p,q} z\,=\,0 \};
\\
&{\rm SO} (p,q):=\, \{g \,\in\, {\rm SL}_{p+q}(\R)\,\mid\, g^t{\rm I}_{p,q} g \,=\, {\rm I}_{p,q} \},\   \
{\s\o} (p,q):= \{z \,\in\, \s\l_{p+q}(\R) \,\mid\, z^t {\rm I}_{p,q} + {\rm I}_{p,q} z \,=\,0 \};
\\
&{\rm Sp} (p,q):=\, \{g \,\in\, {\rm SL}_{p+q}(\H)\, \mid\, \overline{g}^t{\rm I}_{p,q} g \,=\, {\rm I}_{p,q} \},\    \
{\s\p} (p,q):= \{z \,\in\, \s\l_{p+q}(\H) \,\mid\, \overline{z}^t{\rm I}_{p,q} + {\rm I}_{p,q} z \,=\,0 \};
\\
&{\rm Sp} (n,\R):=\, \{g \,\in\, {\rm SL}_{2n}(\R) \,\mid\, g^t {\rm J}_{n} g \,=\,{\rm J}_{n} \}, \quad
{\s\p} (n,\R)\,:=\, \{z \,\in\, \s\l_{2n}(\R) \,\mid\, z^t {\rm J}_{n} + {\rm J}_{n} z \,=\,0 \};
\\
&{\rm SO}^* (2n):=\, \{g \,\in\, {\rm SL}_{n}(\H) \,\mid\, \overline{g}^t \jb{\rm I}_n g\,= \, \jb{\rm I}_n \}, \quad
{\s\o}^* (2n)\,:=\, \{z \,\in\, \s\l_{n}(\H) \,\mid\, \overline{z}^t \jb {\rm I}_n + \jb {\rm I}_n z \,=\,0 \}.
\end{align*}

\subsection{Partitions and (signed) Young diagrams}\label{sec-partition-Young-diagram}
For two   disjoint ordered sets  $(v_1,\, \ldots ,\,v_n)$ and $(w_1,\, \ldots ,\,w_m)$, 
the ordered set $(v_1,\, \ldots ,\,v_n,\, w_1,\, \ldots ,\, w_m)$ will be denoted by
$$
(v_1, \,\ldots ,\,v_n) \vee (w_1,\, \ldots ,\,w_m)\, .
$$
A {\it partition} of a positive integer $n$ is an object of the form
$\d:=[ d_1^{t_{d_1}},\, \ldots ,\, d_s^{t_{d_s}} ]$, where $t_i,\, d_i \,\in\, \N$, $ 1 \,\leq\, i
\,\leq\, s$, such that
$\sum_{i=1}^{s} t_{d_i} d_i \,=\, n$, $t_{d_i} \,\geq\, 1$ and $ d_1>\cdots>d_s>0$;  see   \cite[\S~3.1, p.~30]{CoMc}.
Let $\PC(n)$ be the {\it set of all partitions of $n$}. For
a partition $\d\, =\, [ d_1^{t_{d_1}},\, \ldots ,\, d_s^{t_{d_s}} ]$ of $n$, define 
\begin{equation}\label{Nd-Ed-Od}
	{\N}_{\d} \,:=\, \{  d_1,\, \ldots ,\, d_s \}\, ,\ \
	{\E}_{\d} \,:=\, {\N}_{\d}\cap (2\N)\, , \ \ 
	{\O}_{\d} \,:=\, {\N}_{\d}\setminus {\E}_{\d} .
1\end{equation}
\medskip 
Further define 
\begin{equation}\label{Od1-Od3}
	{\O}^1_{\d} := \{ d\, \mid \,d \in \O_\d, \  d \equiv 1 \pmod{4} \}\ 
	\text{ and } \ {\O}^3_{\d} := \{ d\, \mid\, d \in \O_\d,\  d \equiv 3 \pmod{4} \} .
\end{equation}
\medskip 
Following \cite{CoMc}, a partition
$\d$ of $n$ will be called {\it even} if ${\N}_{\d} \,=\, {\E}_{\d}$. 
Let $\PC_{\rm even} (n)$ be the subset of $\PC(n)$ consisting of all even partitions of $n$. 
We call a partition 
$\d$ of $n$ to be {\it very even} if $\d$ is even, and $t_\eta$ is even for all $\eta \in {\N}_{\d}$.
Let $\PC_{\rm v. even} (n)$ be the subset of $\PC(n)$ consisting of all very even partitions of  $n$.
Now define
\begin{align}\label{defn-Pn-1}
	\PC_{-1}(n) \,:=&\, \{\d \,\in\, \PC(n) \,\mid\,  t_\theta ~\text{ is even for all }~ \theta \,\in\, \O_\d \}\, ,\\
	\PC_{1}(n) \,:=&\, \{\d \,\in\, \PC(n) \,\mid\,  t_\eta ~\text{ is even for all }~ \eta \,\in\, \E_\d \}\, .    \label{defn-Pn1}
\end{align}

Next, we will define the Young diagram and signed Young diagram, see \cite[p. 140]{CoMc}. For our convenience, we will follow the modified definition of the signed Young diagram used in \cite[\S 2.4]{BCM}. Such signed Young diagram is used to parametrize the nilpotent orbits in classical simple real Lie algebras. In Section \ref{sec-reality-nilpotents--real-g}, such notion will be used.

\medskip

\begin{definition}\label{def-young-diag}~
	\begin{itemize}
\item  A {\it Young diagram} is a left-justified array of rows of empty boxes arranged so that no row is shorter than the one below it; the {\it size} of a  Young diagram is the number of empty boxes appearing in it.  There is an obvious correspondence between the set of Young diagrams of size $n$ and the set $\PC(n)$ of partitions of $n$.  Hence the {\it set of Young diagrams of size $n$} is also denoted by $\PC(n)$.
		
\item A {\it signed Young diagram} is a Young diagram in which every box is labelled with $+1$ or $-1$ such	that the sign of $1$ alternate across rows except when the length of the row is of the form  $3\,\pmod{4}$. In the latter case when the length of the row is of the form $3\,\pmod{4}$  
we will alternate the sign of $1$ till the last but one and repeat the sign of $1$ in the last	box as in the last but one box.	The sign of $1$ need not alternate down columns. 
\item  Two	signed Young diagrams are equivalent if and only if each can be obtained from the other by	permuting rows of equal length. 
\item The {\it signature of a signed Young diagram} is the	ordered pair of integers $(p,\,q)$ where $p$-many $+1$ and $q$-many $-1$ occur in it.     \qed
\end{itemize} 	
\end{definition}

\begin{remark}	\label{rmk-signed-young-diag-unique}
Let $ p_d $ (resp. $ q_d $) be the number of $ +1 $ (resp. $ -1 $) in the $ 1^{\rm st} $ column of the rectangular block of size $ t_d\times d $.
Then a signed Young diagram of size $ n $ is uniquely determined by a partition $\d\, =\, [ d_1^{t_{d_1}},\, \ldots ,\, d_s^{t_{d_s}} ]$ of $n$, and a pair of integer $ (p_d,q_d) $ for all $ d\in \N_{\d} $ so that $ p_d+q_d=t_d $.  \qed
\end{remark}

\subsection{Jacobson-Morozov Theorem and associated results}\label{sec-Jacobson-Morozov-partion}
For a  Lie algebra $ \g $ over $\R$, a subset $\{X,H,Y\} \,\subset\, \g$ is said to be a {\it $\s\l_2$-triple}  if $X \,\neq\, 0$, $[H,\, X] = 2X$, $[H,\, Y] =  -2Y$ and $[X,\, Y] =H$. 
Note that  for a $\s\l_2$-triple $\{X,H,Y\}$ in $\g$, $\text{Span}_\R \{X,H,Y\}$ is isomorphic to  $\s\l_2(\R)$. We now recall a famous result due to Jacobson and Morozov.

\begin{theorem}[{Jacobson-Morozov, cf.~\cite[Theorem~9.2.1]{CoMc}}]\label{Jacobson-Morozov-alg}
Let $X\,\in\, \g$ be a non-zero nilpotent element in a  semisimple Lie algebra $\g$ over $\R$. Then there exist $H,\,Y\, \in\, \g$ such that $\{X,H,Y\}$ is a $\s\l_2$-triple.
\end{theorem}

Let $V$ be a right $\D$-vector space of dimension $n$, where $\D$ is, as before, $\R$ or $ \C $ or $ \H$.  Let $\{X,\,H,\,Y\} \,\subset\, \s\l (V )$ be a $\s\l_2$-triple. Note that $V$ is also  a $\R$-vector space using the inclusion $\R \,\hookrightarrow\, \D$. Hence $V$ is a module  over $ \text{ Span}_\R \{ X,\,H,\,Y\} \,\simeq\, \s\l_2(\R)$.   For any positive integer $d$, let $M(d-1)$ denote the sum of all the $\R$-subspaces $W$ of $V$ such that
\begin{itemize}
\item $\dim_\R W \,= \,d$, and
\item $W$ is an 	irreducible $\text{Span}_\R \{ X,H,Y\}$-submodules of $V$.
\end{itemize} 
Then $M(d-1)$ is the {\it isotypical component} of $V$ containing all the irreducible submodules  of $V$ with highest weight $d-1$. 
As  $X,\,H,\,Y$ of $V$ are $\D$-linear, the $\R$-subspaces $M(d-1)$ of $V$ are also $\D$-subspace. Let
\begin{equation}\label{definition-L-d-1}
  L(d-1)\,:= \,{\rm ker\,}Y \cap M(d-1)\, , \quad {\rm  and } \quad t_{d} \,:=\, \dim_\D L(d-1)\, .
\end{equation}
Consider the non-zero irreducible $\text{Span}_\R \{ X, H, Y\}$-submodules of $V$. 
Let $\{d_1,\, \ldots, \, d_s\}$, with $d_1 \,>\, \cdots \,>\, d_s$, be the integers that occur as $\R$-dimension of such $\text{Span}_\R \{ X,\,H,\,Y\}$-modules. Then it follows that $\sum_{i=1}^s t_{d_{i}} d_i \,=\,{\dim}_\D V\,=\, n\, .$ Thus, 
\begin{equation}\label{partition-symbol}
\d \,:=\, \big[d_1^{t_{d_1} },\, \ldots ,\, d_s^{t_{d_s} }\big] \,\in\, \PC(n)\, .
\end{equation}
Recall that for a partition $\d\in \PC(n)$, $\N_\d$, $\E_\d $ and $\O_\d $ are defined in \eqref{Nd-Ed-Od}.

\begin{proposition}[{\cite[Proposition A.2]{BCM}}]  \label{J-basis}
	Let $\{X,H,Y\} \subset \s\l (V) $ be a $\s\l_2$-triple, where $V$ is a right $\D$-vector space of dimension $n$. 
	For all $d \,\in \,\N_\d$  and for any $\D$-basis of 
	$L(d -1)$, say, $\{ v^d_j  \,\mid\,  1 \,\leq\, j \,\leq\, t_d\,:=\, \dim_\D L(d-1) \}$ the following two hold:
	\begin{enumerate}
		\item $ X^d v^d_j \,=\, 0$ and  $H(X^l v^d_j)\,= \, X^lv^d_j(2l + 1 - d) $ for $ 1\,\leq\, j \,\leq\, t_d$,
		$0 \,\leq\, l \,\leq\, d-1$, $d \,\in\, \N_\d$.
		
		\item  For all $d \,\in \,\N_\d$, the set $\{ X^lv^d_j \,\mid\,  1\,\leq\, j \,\leq\, t_d,\  0 \,\leq\, l
		\,\leq\, d-1 \}$ is a $\D$-basis of $M(d-1)$. In particular, 
		$\{ X^l v^d_j \,\mid  \, 1\,\leq\, j \,\leq\, t_d, \  0 \,\leq\, l \,\leq\, d-1,\ d \,\in\, \N_\d \}$
		is a $\D$-basis of $V$. 
	\end{enumerate}
\end{proposition}

\begin{proposition}[{\cite[Proposition A.6]{BCM}}]\label{unitary-J-basis}
Let $V$ be a right $\D$-vector space of dimension $n$, $\epsilon \,= \,\pm 1$, $\sigma \,\colon\,  \D \,\longrightarrow\, \D$ is either the identity map or it is $\sigma_c$ when $\D$ is $\C$ or $\H$.  Let $\langle \cdot,\, \cdot \rangle \,:\, V \times V \,\longrightarrow\, \D$ be a $\epsilon$-$\sigma$ Hermitian form. Let $\{X,\,H,\,Y\}\,\subset\, \s\u (V,\, \langle \cdot,\, \cdot \rangle ) $ be a $\s\l_2$-triple.  
Let $d \,\in\, \N_\d$ and $t_{d}\,:= \,\dim_\D L(d-1)$.  Then for all $d \,\in\, \N_\d$, there exists a $\D$-basis $\{ {v^d_j}\,\mid \, 1\,\leq\, j \,\leq \,t_d \}$ of  $L(d -1)$ such  that the following three hold:
\begin{enumerate}
\item $X^d v^d_j \,=\, 0$ and $H(X^l v^d_j)\,=\,  X^l v^d_j(2l + 1 - d)$ for all
		$1\,\leq\, j \,\leq\, t_d$, $0 \,\leq\, l \,\leq\, d-1$ and $d \,\in\, \N_\d$.
		
\item  For all $d \,\in\, \N_\d$, the set $\{ X^lv^d_j\,\mid\,  1\,\leq\, j \,\leq\, t_d,\ 0 \,\leq\, l
		\,\leq\, d-1 \}$ is a $\D$-basis of $M(d-1)$. In particular, 
		$\{ X^l v^d_j \,\mid\, 1\,\leq\, j \,\leq\, t_d,\  0 \,\leq\, l \,\leq\, d-1,\ d \,\in\, \N_\d\}$
		is a $\D$-basis of $V$. 
		
		\item  The value of $\langle \cdot  , \,\cdot \rangle$ on any pair of the above basis vectors is $0$, except in the
		following cases: 
\begin{itemize} 
    \item  If $\sigma \,=\, \sigma_c $, then $ \langle X^l v^d_j,\, X^{d -1 -l} v^d_j \rangle \,\in\, \D^*$.
			
			\item If $\sigma \,=\, {\rm Id}$ and $\epsilon \,=\,1$, then $\langle X^l v^d_j ,\, X^{d -1-l} v^d_j \rangle \,\in\, \D^*$ 
			when $d \,\in\, \O_\d$, and 
			$\langle X^l v^d_j ,\, X^{d-1 -l} v^d_{j+1} \rangle \,\in\, \D^* $ when $ d \,\in\,\E_\d $ and $j$ is odd. 
			
			\item If $\sigma \,= \,{\rm Id}$ and $\epsilon \,= \,-1$, then
			$\langle X^l v^d_j ,\, X^{d -1-l} v^d_j \rangle \,\in\, \D^*$ when $d \in \E_\d$, and  $\langle X^l v^d_j ,\, X^{d-1 -l} v^d_{j+1} \rangle \,\in\, \D^* $ when $d\,\in\,\O_\d $ and $j$ is odd.
		\end{itemize}
	\end{enumerate}
\end{proposition}

In the following remark, we will  simplify the basis of $V$  as in Proposition \ref{unitary-J-basis} case by case which will be used later. 
We  need to consider only the following cases:

\begin{remark}[{\cite[Remark A.8]{BCM}}]\label{unitary-J-basis-rmk}	
	The $\D$-basis elements  $\{ v^d_j \,\mid\,  1\,\leq\, j\, \leq\, t_d \}$ of $L(d-1)$ in Proposition \ref{unitary-J-basis}	are modified as follows:
	\begin{enumerate} 
		\item 
		If $\D\,=\,\R$ and  $\epsilon \,=\,1$, by suitable rescaling each element of
		$\{ v^d_j \,\mid\,  1\,\leq\, j\, \leq\, t_d \}$ we may assume that 
		\begin{itemize}
			\item  $\langle v^d_j,\, X^{d-1} v^d_j \rangle\,=\, \pm 1$ when
			$d \,\in\, \O_\d$, and
	
\item $\langle v^d_j,\, X^{d-1} v^d_{t_d/2 +j} \rangle \,=\, 1 $ when
$d \,\in\, \E_\d$ and $1\,\leq\, j\,\leq\, t_d/2$.
	\end{itemize}
		If $\D\,=\,\R$ and  $\epsilon \,=\,-1$, analogously we may assume that the elements of
		the $\R$-basis  $\{ v^d_j \,\mid\,  1\,\leq\, j \,\leq \,t_d \}$ of $L(d-1)$ in Proposition \ref{unitary-J-basis}
		satisfy the condition that 
		\begin{itemize}
			\item $\langle v^d_j,\, X^{d-1} v^d_j \rangle \,=\, \pm 1$ when
			$d \,\in \,\E_\d$, and
			
			\item $\langle v^d_j,\, X^{d-1} v^d_{t_d/2 +j} \rangle \,=\, 1 $ when
			$d \,\in\, \O_\d$ and $1\,\leq\, j\,\leq\, t_d/2$.
		\end{itemize}
		
\item 
If $\D\,=\,\C$, $\epsilon \,=\,1$ and $ \sigma \,=\,\sigma_c$, rescaling the elements of the $\C$-basis
    $\{ v^d_j \,\mid \, 1\,\leq\, j \,\leq\, t_d \}$ we may assume that
\begin{itemize}
\item $\langle v^d_j ,\, X^{d-1} v^d_j \rangle\,=\, \pm 1$ when 
			$d\,\in\, \O_\d$, and
			
    \item $\langle v^d_j ,\, X^{d-1} v^d_j \rangle \,=\, \pm \sqrt{-1}$ when $d \,\in\,  \E_\d$.
		\end{itemize}
		
		\item 
		If $\D\,=\, \H$, $\epsilon\, =\,1$ and $\sigma \,=\, \sigma_c$, after rescaling and conjugating 
		the elements of the $\H$-basis $\{ v^d_j \,\mid\,  1\,\leq\, j \,\leq\, t_d \}$ of $L(d-1)$
		by suitable scalars the elements of the $\H$-basis, we may assume that 
		\begin{itemize}
			\item $\langle v^d_j,\, X^{d-1} v^d_j \rangle\,=\, \pm 1$ when $d \,\in\, \O_\d$, and
			
			\item $\langle v^d_j,\, X^{d-1} v^d_j \rangle\,=\, \jb$ when $d \,\in\, \E_\d$.
		\end{itemize} 
		If $\D\,=\,\H$, $\epsilon \,=\, -1$ and $\sigma  \,=\, \sigma_c$, analogously we may assume that the elements of
		the $\H$-basis $\{ v^d_j \,\mid\,  1\,\leq\, j \,\leq \,t_d \}$ of $L(d-1)$
		satisfy
		\begin{itemize}
			\item $\langle v^d_j, \,X^{d-1} v^d_j \rangle\,=\, \pm 1$ when $d \,\in\, \E_\d$, and
			
			\item $\langle v^d_j,\, X^{d-1} v^d_j \rangle\,= \,\jb$ when $d \,\in\, \O_\d$.   \qed
		\end{itemize}
	\end{enumerate}
\end{remark}

The following  basic result will be used to prove Theorem \ref{thm-spnC-st-real}.

\begin{lemma}[{cf.~\cite[Lemma 2.6]{BCM}}] \label{comm-XH} 
	Let $\{X,\,H,\,Y\}$ be a $\s\l_2$-triple in the Lie algebra $\g$
	of a Lie group $G$. Then $Z_{G}(X,H) \,= \,Z_{G}(X,H,Y)$.
\end{lemma}

Next, we will state a well-known result of Springer-Steinberg regarding the reductive part of the  centralizer of a nilpotent element in the Lie group $ {\rm Sp}(n,\C) $.
\begin{theorem}[{Springer-Steinberg, cf.~\cite[Theorem 6.1.3]{CoMc}}] \label{thm-springer-steinberg}
Let $X $ be a nilpotent element in $ \s\p(n,\C) $ which corresponds to the partition $ \d=[d_1^{t_{d_1}}, \ldots, d_s^{t_{d_s}}] \in \PC_{-1}(2n)$.  Let $ \{X,H,Y\} $ be a $ \s\l_2 $-triple in $ \s\p(n,\C) $ containing $ X $. Then
	$$
Z_{{\rm Sp}(n,\C)}(X,H,Y) \,\cong\,   \prod_{\eta \in \E_\d} {\rm O}(t_\eta) \times  \prod_{\th \in \O_\d} {\rm Sp}(t_\th, \C)	\,. 
	$$ 
\end{theorem}

Next result gives the reductive part of the  centralizer of a nilpotent element in the Lie group $ {\rm Sp}(p,q) $.
\begin{proposition}[{\cite[Proposition 3.34]{BCM2}}] \label{cor-centralizer-sl2-triple-sp-pq}
Let $ X\in \s\p(p,q) $ be a	non-zero nilpotent element, and $ \{X,H,Y\} $ be a $ \s\l_{2} $-triple in $ \s\p(p,q) $. Let the nilpotent orbit $ \OC_X$ corresponds to the signed Young diagram of signature $ (p,q) $ where $ p_\theta $ (resp. $ q_\th $) denotes the number of $ +1 $  (resp. $ -1 $) in the $ 1^{\rm st} $ column of the block of size $ t_\th \times \th $ for $ \th\in \O_\d$. Then
	$$
Z_{{\rm Sp}(p,q)}(X,H,Y) \simeq \prod_{\eta \in \E_\d} {\rm SO^*}(2t_\eta) \times  \prod_{\th \in \O_\d} {\rm Sp}(p_\th, q_\th)	\,. 
	$$
\end{proposition}

 Next we include  another basic result which will be needed to prove Theorem \ref{thm-sl-n-R-nilpot}.
\begin{lemma}[{\cite[Lemma 4.1.1 ] {Ma}}]\label{comm-Xd-even} 
Let $ X\in \s\l_n(\R) $ be a non-zero nilpotent element so that the corresponding  partition $ \d\in  \PC_{\rm even} (n)$. Suppose that  $gX=Xg$ for some $ g\in {\rm GL}_n(\R)$. Then $ \det g >0 $.
\end{lemma}

\subsection{An ordered basis of $V$}\label{sec-ordered-basis-V}
Here we will fix a suitable ordering of the basis as described in Propositions \ref{J-basis}, \ref{unitary-J-basis} of the underlying vector space $ V$. This ordering is motivated by a similar ordering used in \cite[(4.4), p. 738]{BCM}. After starting with their initial ordering, we will  modify it to serve our purpose. 
In the following example we will construct an explicit ordered basis of the underlying vector space and express any linear map $ g $  satisfying $\epsilon X=g X g^{-1}$ where $ \epsilon =\pm1 $ with respect to this ordered basis.

\begin{example}
Let $ X\in \s\l(7,\C) $ be a nilpotent element which corresponds to the partition $ [3, 2^2] $ of $ 7$. There exists $ v_1,v_2,v_3\in\C^7 $ such that the vector space $ \C^7 $ has a basis of the form $ \{v_1, Xv_1, X^2v_1\,; \,v_2, Xv_2\,; \,v_3, Xv_3 \} $. We now re-ordered the above basis as follows:
\begin{align}\label{basis-B-exmp}
\BC:=\{ X^2v_1\,;\,  Xv_2, Xv_3\,;\, Xv_1\,;\, v_2, v_3\,;\, v_1  \} \,. 
	\end{align}
	If  $g\in {\rm GL}(\C^7)$ be so that  $\epsilon X=g X g^{-1}$ where $ \epsilon =\pm1 $. Note that  the linear map 
	 $ g$ is completely determined by it's action on $ v_r$ for $ r=1,2,3 $. Write $ g (v_r)\,=\, \sum_{i,\,l} c^l_{i\,r} X^lv_i$ for $ r=1,2,3 $, where $ c^l_{i\,l} \in \C$. i.e.,  
	\begin{align*}
		&~ g (v_1)\,=  \, c^2_{1\,1}X^2v_1 \, + \,  c^1_{2\,1} Xv_2 \,  + \,  c^1_{3\,1} Xv_3 \, + \, c^1_{1\,1} Xv_1 \, +\,  c^0_{2\,1} v_2 \, + \,  c^1_{3\,1}  v_3\, + \, c^0_{1\,1} v_1\,,\\
		& ~ g (v_2)\,=\, c^2_{1\,2}X^2v_1 \, + \,  c^1_{2\,2} Xv_2 \,  + \,  c^1_{3\,2} Xv_3 \, + \, c^1_{1\,2} Xv_1 \, +\,  c^0_{2\,2} v_2 \, + \,  c^1_{3\,2}  v_3\, + \, \magenta{c^0_{1\,2}} v_1\,,\\
		& ~ g (v_3)\,=\, c^2_{1\,3}X^2v_1 \, + \,  c^1_{2\,3} Xv_2 \,  + \,  c^1_{3\,3} Xv_3 \, + \, c^1_{1\,3} Xv_1 \, +\,  c^0_{2\,3} v_2 \, + \,  c^1_{3\,3}  v_3\, + \, \magenta{c^0_{1\,3}} v_1\,,\\
	\end{align*}
Now using the relations $ \epsilon X =g X g^{-1}$, and $ X^3v_1=0,\,X^2v_i=0,\, i=2,3$; it follows that 
\begin{itemize}
		\item $~ g (X v_1)\,= \, \epsilon X g (v_1) \, = \epsilon  (\, c^1_{1\,1} X^2v_1 \, +\,  c^0_{2\,1} Xv_2 \, + \,  c^1_{3\,1}  Xv_3\, + \, c^0_{1\,1} Xv_1)$\,,
		\item $~ g (X v_2)\,=\, \epsilon X g (v_2) \, = \epsilon (\, c^1_{1\,2} X^2v_1 \, +\,  c^0_{2\,2} Xv_2 \, + \,  c^1_{3\,2}  Xv_3\, + \, \magenta{c^0_{1\,2}} Xv_1)$\,,
		\item $~ g (X v_3)\,= \,\epsilon  X g (v_3) \, = \epsilon  (\, c^1_{1\,3} X^2v_1 \, +\,  c^0_{2\,3} Xv_2 \, + \,  c^1_{3\,3}  Xv_3\, + \, \magenta{c^0_{1\,3}} Xv_1)$\,,
		
		\item $~ g (X^2 v_1)\,=\, c^0_{1\,1} X^2v_1$.
	\end{itemize}
	Since $g (X^{2} v_2) = g (X^{2} v_3) =0$,  we have $ \magenta{ c^0_{1\,2}  = c^0_{1\,3} =0 }$. 
	Observe that 
	the matrix $ [g]_\BC $ is of the form :
	$$
	[g]_\BC\,=\,
	\begin{pmatrix}
		\blue{c^0_{1\,1}} & \epsilon c^1_{1\,2}  & \epsilon c^1_{1\,3} &  \epsilon c^1_{1\,1} &   c^2_{1\,2} &    c^2_{1\,3} & c^2_{1\,1}     
		\\
		& \violet{ \epsilon c^0_{2\,2}} & \violet{ \epsilon c^0_{2\,3}} & \epsilon c^0_{2\,1} &  c^1_{2\,2}  & c^1_{2\,3}  & c^1_{2\,1}  
		\\
		&\violet{ \epsilon c^0_{3\,2}}& \violet{ \epsilon c^0_{3\,3}}& \epsilon c^0_{3\,1} & c^1_{3\,2}  & c^1_{3\,3}  &  c^1_{3\,1} 
		\\  
		&   & &  \blue{\epsilon c^0_{1\,1}}  & c^2_{1\,2} & c^1_{1\,2} & c^1_{1\,1}    
		\\  
		& &  &   &\violet{  c^0_{2\,2}}  & \violet{ c^0_{2\,3}} & c^0_{2\,1}
		\\
		&  &   &   &\violet{  c^0_{3\,2}} & \violet{ c^0_{3\,3}} & c^0_{3\,1}
		\\
		& &  &   &   &  &  \blue{c^0_{1\,1}} 	  
	\end{pmatrix} \,.$$
	\qed
\end{example}

Now we will consider the general case. 
Let $(v^d_1,\, \ldots ,\, v^d_{t_d})$  be the ordered $\D$-basis of $L(d-1)$ as in Proposition \ref{unitary-J-basis}  for $d \,\in\, \N_\d$.  Then, as done in  \cite[(4.3)]{BCM}, it follows from Proposition \ref{unitary-J-basis} that $\BC^l (d) \,:=\, (X^l v^d_1,\, \ldots ,\,X^l v^d_{t_d})$ is an ordered $\D$-basis of $X^l L(d-1)$ for $0\,\leq\, l \,\leq\, d-1$ with $d\,\in\, \N_\d$. 
Define
\begin{equation}\label{old-ordered-basis}
\BC(j) \,:=\, \BC^{d_1-j} (d_1) \vee    \cdots \vee  \BC^{d_s-j} (d_s)  ,\ \text{ and }\  \BC\,:=\, \BC(1) \vee \cdots \vee  \BC(d_1)\, .
\end{equation}
Then $\BC$ is an ordered basis of  $V$.  Note that
$$
\BC(1)=(X^{d_1-1} v_1^{d_1}, \ldots ,\,X^{d_1-1}v^{d_1}_{t_{d_1}}; X^{d_2-1} v_1^{d_2}, \ldots,\,X^{d_2-1}v^{d_2}_{t_{d_2}};\dots\ldots ; X^{d_s-1} v_1^{d_s}, \ldots ,\,X^{d_s-1}v^{d_s}_{t_{d_s}}),
$$
$$
\BC(d_1)=(v_1^{d_1},\, \ldots ,\, v^{d_1}_{t_{d_1}})  .
$$ 
Note that $X^{d_i} v_j^{d_i} \,=\, 0 $ for $ i=1,\dots, s;\, j=1, \dots, t_{d_i}$.  
Thus $X|_{{\rm Span}\BC(1)}=0$.   Suppose $\tau \in {\rm GL}(V)$  so that  $ \tau X \tau^{-1} =\epsilon  X $, where $\epsilon = \pm1$. Since $ \tau X^n\,=\, \epsilon^n X^n\tau $ for all $ n\in \N $, the  following flag
$$\{0\} \, \subset\, \ker X \, \subset \, \ker X^2 \, \subset \, \cdots \, \subset \ker X^{d_1} =V\,$$
 remains invariant under $ \tau $. Note that $\ker X^t = {\rm Span}\{\BC(1)\vee \cdots \vee \BC(t) \}  $  for $t=1,\ldots , d_1$. 	Following the above notation, we conclude that 
	\begin{enumerate}
		\item  $ \tau  \big({\rm Span}\{\BC^{d_1-1}(d_1)\vee \cdots\vee \BC^{d_i-1}(d_i) \}\big) \,\subseteq \, {\rm Span}\{\BC^{d_1-1}(d_1)\vee \cdots\vee \BC^{d_i-1}(d_i) \} $\\  for $ i=1,\ldots, s $\,;
		\item $ \tau \big( {\rm Span}\{\BC(1)\vee \cdots \vee \BC(t) \}\big) \,\subseteq \, {\rm Span}\{\BC(1)\vee \cdots \vee \BC(t) \}$ for $t=1,\ldots , d_1$. 
	\end{enumerate}

Also,  $\tau $ will be determined uniquely by it's action on $ L(d-1) $ for $ d\in \N_\d $.
The matrix  $[\tau]_\BC$ is a block upper triangular matrix with $ (d_1+\cdots +d_s) $-many diagonal blocks. Write $ [\tau]_\BC = (A_{ij}) $, where $A_{ij}$ are  block matrices. 
Then $ A_{ij} =0$ for $ i>j $. 
The order of  the first $s$-many diagonal blocks $A_{jj}$   of $[\tau]_\BC$, is $ t_{d_j}\times t_{d_j} $, where $ t_{d_j}=\dim L(d_j-1) $ for $j=1,\ldots , s$. 
Any diagonal block of $[\tau]_\BC$ is of the form $\epsilon A_{jj}$  for  $j=1,\ldots , s$. 
Moreover, if $\tau^2 = {\rm Id}$, then necessarily $A_{jj}^2 ={\rm Id}$ for $j=1,\ldots ,s$.

\section{Reality for nilpotent elements in simple Lie algebra over $\C$}\label{sec-reality-nilpotent-elts-complex-g}
 Using the  Jacobson-Morozov Theorem \ref{Jacobson-Morozov-alg} and Mal'cev Theorem  \cite[Theorem 3.4.12]{CoMc}, it follows that nilpotent elements are real in simple Lie algebra over $ \C $.
We will provide an elementary proof without using those results and also construct an explicit conjugating element for the nilpotent element. The following basic lemma is useful for this.
\begin{lemma}\label{X-gx}
Let $X\in \g $ be a  non-zero nilpotent element in a  semi-simple Lie algebra $\g$ over $ \C $. Then $ X\in [\g, X] $.
\end{lemma}

\begin{proof}
See \cite[$(3.3.4)$, p. 38]{CoMc} for a proof.
\end{proof}

\begin{lemma}\label{thm-real-nilp-g/C}
Every unipotent element in a complex semi-simple Lie group $G$  is real.
\end{lemma}

\begin{proof}
It is enough to prove that every nilpotent element is real in $ \g $, see Corollary \ref{cor-Ad-real-unipotent-elt}. 
Let $ X\in \g $ be a non-zero nilpotent element. Using Lemma \ref{X-gx}, we have $  X=[H,X] $ for some  $ H\in \g$. 	Let $ \alpha\in \C $ so that $ e^{\alpha}=-1 $. Then   $ [\alpha H, X]=\alpha X $ and
$$
{\rm Ad}\,(\exp \alpha H)  (X) = e^{ {\rm ad}({ \alpha H})}(X) = \sum_n \frac{\big({\rm ad}\,({\alpha H})\big)^n}{n!}X  = \sum_n \frac{\alpha ^n}{n!}X  = -X \,.
$$
Hence $ X $ and $( -X) $ are conjugate via $ \exp \alpha H$. This completes the proof.
\end{proof}

An analogous result may not true over $ \R $, see Example \ref{exp-sl2-R}.
\begin{example}\label{exp-sl2-R}
Let us consider the  nilpotent element $ X=\begin{pmatrix}
	0&1\\
	0&0
\end{pmatrix}$. A simple computation shows that $ X $ is not real in $ \s\l_2(\R) $, real in $ \s\l_2(\C) $ but not strongly real in  $ \s\l_2(\C) $, and finally, strongly real in  $ \g\l_2(\C) $.  \qed
\end{example}

 In view of the Lemma \ref{thm-real-nilp-g/C}, we want to classify the strongly real nilpotent elements in a simple Lie algebra over $ \C $ which will be done in the rest of the section.
Moreover, we will construct an explicit  element in $ G $ corresponding to the nilpotent element $ X \in \g$ which conjugates $ X $ and $ -X $.

\subsection{Nilpotent elements in $ \s\l_n(\C) $}
 First we will consider the Lie algebra $ \g\l_n(\D)$ where  $ \D=\R, \, \C,\, \H $. 

\begin{proposition}\label{Propo-glnD}
Every nilpotent element in $ \g\l_n(\D) $ is strongly real when $ \D=\R, \, \C,\, \H $.
\end{proposition}
\begin{proof} 
Let $X\in \g\l_n(\D)$ be a non-zero nilpotent element. Then $X\in \s\l_n(\D)$. Using Proposition \ref{J-basis}, $\D^n$ has a basis of the form $\{ X^l v^d_j \,\mid  \, 1\,\leq\, j \,\leq\, t_d, \  0 \,\leq\, l \,\leq\, d-1,\ d \,\in\, \N_\d \}$. Now define $g\in {\rm GL}_n(\D)$ as follows:
\begin{align}\label{defn-g-matrix}
		g(X^lv^d_j)\,=\, \begin{cases}
			(-1)^lX^lv^d_j & d\in \E_\d\cup \O_\d^1\,,\\
			 	(-1)^{l+1}X^lv^d_j & d\in \O_\d^3  \,.\\
		\end{cases}
\end{align}
Then $g^2 = {\rm Id}$ and $Xg=-gX$. This completes the proof.	
\end{proof}

Now we will consider the nilpotent elements in $\s\l_n(\C) $.  Every nilpotent element is  real in ${\s\l}_n(\C)$, see Lemma  \ref{thm-real-nilp-g/C}. In Example \ref{exp-sl2-R}, we have observed that any nilpotent element in $ \s\l_2(\C) $ is not strongly real.  In general every nilpotent element in ${\rm Lie(\, PSL}_n(\C))$  is  strongly real, but  not in $\s\l_n(\C)$.
The following terminology is needed  for the next result.
 Define  $\E_\d^2 :=\{\eta \in \E_\d\mid \eta \equiv 2\pmod 4  \}$, where $ \E_\d$ is as in \eqref{Nd-Ed-Od}, and 
 \begin{align}\label{def-Pe-n}
 \widetilde{\PC_{e}}(n)\, :=\,  \{\d\in \PC_{even}(n) \setminus\PC_{v.even}(n) \mid \sum_{\eta \in \E^2_\d} t_\eta \ {\rm is \ odd \ } \}\,. 	
 \end{align}

\begin{theorem}
 Every nilpotent element is strongly  real in ${\rm Lie(\, PSL}_n(\C))$.
\end{theorem}

\begin{proof}
Let $X\in \s\l_n(\C)$ be a non-zero nilpotent element. Using Proposition \ref{J-basis}, $\C^n$ has a basis of the form $\{ X^l v^d_j \,\mid  \, 1\,\leq\, j \,\leq\, t_d, \  0 \,\leq\, l \,\leq\, d-1,\ d \,\in\, \N_\d \}$. 
We now divide the proof in two parts:

{\it Case 1:} $\d\not\in \widetilde{\PC_{e}}(n)$. Set $g\in {\rm GL}_n(\C)$ as in \eqref{defn-g-matrix}.
Then $\det g =1$, $g^2 = {\rm Id}$,\, $Xg=-gX$.

{\it Case 2:}  $\d\in \widetilde{\PC_{e}}(n)$.  In this case $n\equiv 2 \pmod 4$. 
  Let $g\in {\rm GL_n(\C)}$ be as in \eqref{defn-g-matrix}, and $\det g =-1$. Then define $\widetilde g:= \sqrt{-1} g $.
Then $\det \widetilde g= (\sqrt{-1})^n \det g =1$,\, $\widetilde g^2 = {\rm -Id}$ and $X\widetilde g=-\widetilde gX$. This completes the proof.
\end{proof}

\begin{theorem} \label{thm-sl-n-c-nilpot}
Let $ X\in \s\l_n(\C) $ be a nilpotent element, and $ \d \in \PC(n)$ be the corresponding partition. Then $ X $   is strongly real  if and only if  the partition $ \d \not\in  \widetilde{\PC_{e}}(n)$.
\end{theorem}

\begin{proof}
Let $ X \in \s\l_n(\C)$ be a non-zero nilpotent element, and  $ \{X,H,Y\} $ be a $\s\l_2$-triple in $ \s\l_n(\C)$ containing $ X $. Then $\C^n$ has a basis of the form $\{X^lv_i\mid i,l\} $, see Proposition \ref{J-basis}. First assume that $ \d \not\in  \widetilde{\PC_{e}}(n)$. In this case, define an involution $g$ as in \eqref{defn-g-matrix}. Since the corresponding partition $ \d\not\in \widetilde{\PC_{e}}(n) $, it follows from the construction as done in \eqref{defn-g-matrix} that $ \det g \,=\, (-1)^{\sum_{\eta \in  \E_\d^2} t_\eta}\,=\,1$.  This shows that  $ X $ is strongly real.

Next assume that  $X$ is strongly real, and $ \d\in  \widetilde{\PC_{e}}(n)$. The converse part  of this theorem will follow if we arrive a contradiction.
Let  $ \tau \in {\rm SL}_n(\C)$ be an involution so that $  \tau X  \tau^{-1} =-X$. We will write $ \tau $ with respect to the ordered basis $ \BC $ constructed in \eqref{old-ordered-basis} so that $ [ \tau ]_\BC $ is block upper triangular matrix and follow the notation of \S \ref{sec-ordered-basis-V}. Since, $ \tau^2 ={\rm Id}$,  $ A_{jj}^2={\rm Id}$ for $ 1\leq j \leq s $. It follows form the definition of the ordered basis $ \BC $, that $ A_{jj}$ and  $ -A_{jj}$ both  occur $ d_j/2 $-many times in the diagonal block for $ 1\leq j \leq s$.  This contradicts the fact that $ \det \tau=1 $, as
$$
1\,=\,\det \tau \,= \, (-1)^{\sum_{\eta \in  \E_\d^2} t_\eta}(\det A_{jj})^{d_j} \,=\,(-1)^{\sum_{\eta \in  \E_\d^2} t_\eta}\,=\,-1\,. 
$$
This completes the proof.
\end{proof}

\begin{remark}
 Similar proof will work for $\s\l_n(\R)$. But we will provide more straightforward proof in Theorem \ref{thm-sl-n-R-nilpot} for $\s\l_n(\R)$.
\end{remark}

\subsection{Nilpotent elements in $ \s\o(n,\C) $}
 Throughout this subsection $ \<> $ denotes the symmetric form on $ \C^n $ defined by $ \langle x,y \rangle = x^ty $ for $ x,y\in \C^n $.   Recall the definition of $ \PC_{1} (n)$ as in \eqref{defn-Pn1}.

\begin{theorem}
Every nilpotent element is strongly real in $ \s\o(n,\C) $ for $ n\geq 2 $. 	
\end{theorem} 

\begin{proof}
Let $ X\in \s\o(n,\C) $	be a non-zero nilpotent element. Let $ \{X,H,Y\} $ be a $ \s\l_2 $-triple in $ \s\o(n,\C) $.  Let  $\{ X^l v^{d}_j \,\mid  \, 1\,\leq\, j \,\leq\, t_{d}, \  0 \,\leq\, l \,\leq\, d-1 , d\in \N_{\d}\}$  is a $\C$-basis of $\C^n$ as in Proposition \ref{unitary-J-basis}. The partition $\d\in \PC_1(n)$ for $X$, see \cite[Theorems 5.1.2, 5.1.4]{CoMc}. Define $ g\in {\rm GL}(V) $ as follows:  
$$  
g(X^l v^{d}_j ) := \begin{cases}
(-1)^l X^{l} v^{d}_j &    \text{if }   d\in \O_\d^1\,, 	\\
(-1)^{l+1} X^{l} v^{d}_j & \text{if }  d\in \O_\d^3\,,  	\\
(-1)^{l} X^l v^{d}_{j+t_d/2} & \text{if }  d\in \E_\d\,, \, 1\leq j\leq t_d/2\,,\\
(-1)^{l} X^l v^{d}_{j-t_d/2} & \text{if } d\in \E_\d\,, \,  t_d/2< j\leq t_d \,. 
\end{cases}
$$
It follows from the definition  that $ gX=-Xg $, $ g^2 = 1 $, and $ \det g =1$.  Moreover,
\begin{align*}
\langle g(X^l v^{d}_j ), \,g(X^{d-l-1} v^{d}_j ) \rangle &= \langle X^l v^{d}_j ,\, X^{d-l-1} v^{d}_j  \rangle \qquad\qquad \text{ for } d\in \O_\d \,, \\
	\langle g(X^l v^{ {d}}_j ), \,g(X^{ {d}-l-1} v^{ {d}}_j ) \rangle &= \langle X^l v^{ {d}}_{j+t_d/2} ,\, X^{ {d}-l-1} v^{ {d}}_{j+t_d/2}  \rangle \quad  \text{for }  {d} \in \E_\d\,, 1\leq j\leq t_d/2\,. 
\end{align*}
In view of Proposition \ref{unitary-J-basis}, $\langle gx,gy \rangle = \langle x,y \rangle$ for all $ x,y\in \C^n $. Hence, $ g\in {\rm SO}(n,\C)$. 
\end{proof} 

\subsection{Nilpotent elements in $ \s\p(n,\C) $}
Throughout this subsection $ \<> $ denotes the skew symmetric form on $ \C^{2n} $ defined by $ \langle x,y \rangle = x^t{\rm J}_ny $ for $ x,y\in \C^{2n} $, where $ {\rm J}_n $ is as in  \eqref{defn-I-pq-J-n}. Recall the definition of $ \PC_{-1} (2n)$ as in \eqref{defn-Pn-1}.

In the case of $ \s\p(n,\C) $, every nilpotent element is not strongly real. For example,  any nilpotent element in  $ \s\p(1,\C) $ is not strongly real.  Nilpotent orbits in $ \s\p(n,\C) $ are parametrized by the partition of $ 2n $ in which odd parts occurs with even multiplicity, see \cite[Theorem 5.1.3]{CoMc}. Let $ X\in \s\p(n,\C) $ be a nilpotent element which corresponds to the partition $ \d:=[d_1^{t_{d_1}}, \ldots, d_s^{t_{d_s}}] \in \PC_{-1}(2n)$.  
Recall that in view of Lemma \ref{thm-real-nilp-g/C}, every  nilpotent element in $ \s\p(n,\C) $ is  real.  
Here we will construct an element $g\in {\rm Sp}(n,\C)$ so that $gX=-Xg$ for nilpotent $ X\in \s\p(n,\C) $. 
Let  $\{ X^l v^{d}_j \,\mid  \, 1\,\leq\, j \,\leq\, t_{d}, \  0 \,\leq\, l \,\leq\, d-1, d\in\N_{\d}\}$  is a $\C$-basis of $\C^{2n}$ as in Proposition \ref{unitary-J-basis}. Define $ g\in {\rm GL}\,(\C^{2n} )$ as follows:
\begin{align}\label{real-nilpotent-sp-n-c}
g(X^l v^{d}_j ) := \begin{cases}
(-1)^l X^{l} v^{d}_j &  \qquad   \text{\ if\ } d\in \O_\d   \,,   \\
(-1)^{l}\sqrt{-1} X^l v^{d}_{j} &\qquad \text{\ if\ $ d\in \E_\d$}\, . 
\end{cases}		
\end{align}
Note that $ gX=-Xg $. Using Proposition \ref{unitary-J-basis}, we have $\langle gx,gy \rangle = \langle x,y \rangle$ for all $ x,y\in \C^{2n} $. This shows that $ g\in {\rm Sp}(n,\C)$.   

Recall that for a partition $ \d:=[d_1^{t_{d_1}}, \ldots, d_s^{t_{d_s}}] $, $ t_{d_i} $ is the multiplicity of $d_i\in \N_\d$.

\begin{theorem}\label{thm-spnC-st-real}
A  nilpotent element $ X $ in $ \s\p(n,\C) $ is strongly real  if and only if   $ t_{\eta} $  is even for all  $ \eta \in \E_{\d} $.
\end{theorem} 

\begin{proof}
Let  $\{ X^l v^{d}_j \,\mid  \, 1\,\leq\, j \,\leq\, t_{d}, \  0 \,\leq\, l \,\leq\, d-1, d\in \N_\d\}$  is a $\C$-basis of $\C^{2n}$ as in Proposition \ref{unitary-J-basis}. Suppose that $ t_i $ is even for  even $ d_i $. Define $ g\in {\rm GL}\,(\C^{2n} )$ as follows:
$$  
g(X^l v^{d}_j ) := \begin{cases}
	(-1)^l X^{l} v^{d}_j &  \qquad       \text{\ if\ }  d\in \O_\d    ,\\
(-1)^{l}\sqrt{-1} X^l v^{d}_{j+t_d/2} &\qquad \text{\ if\ $ d\in \E_{\d} $}  , \ 1\leq j\leq t_d/2\,,\\
(-1)^{l+1}\sqrt{-1}  X^l v^{d}_{j-t_d/2} &\qquad \text{\ if\ $ d\in \E_\d $},  \, \  t_d/2< j \leq  t_d\,. 
\end{cases}
$$ 
It follows that $ gX=-Xg $, $ g^2 = 1 $. 
In view of Proposition \ref{unitary-J-basis}, $\langle gx,gy \rangle = \langle x,y \rangle$ for all $ x,y\in V $. This shows that $ g\in {\rm Sp}(n,\C)$.   
	
Next assume that $X$ is a strongly real nilpotent element in $\s\p(n,\C)$. i.e. $ \nu X  \nu^{-1} =-X$  for some involution $  \nu\in {\rm Sp}(n,\C)$. For any $g\in {\rm Sp}(n,\C)$ so that $ gXg^{-1}=-X $,  
\begin{align}\label{involution-tau}
  \nu\, \in\, g Z_G(X)\,.
\end{align}
 Let  $\tau \in Z_G(X)$.   Recall that $ \BC $ is the ordered basis of  $ V $ as in \eqref{old-ordered-basis}. Following \S \ref{sec-ordered-basis-V}, let $ [\tau]_\BC= (A_{ij}) $, and 
 $\tau_D$ be the block diagonal part of  $ [\tau]_\BC$, i.e., consists of   only the matrices $(A_{jj})$ in the diagonal.
  Then $ \tau_D X=X \tau_D$ and $ \tau_D H=H \tau_D$.   Using Lemma \ref{comm-XH}, it follows that $ \tau_D\in Z_G(X,H,Y)$. 
 Because of Theorem \ref{thm-springer-steinberg}, we conclude  that for $1\leq j\leq s$
 \begin{align}\label{block-matrixAii}
  A_{jj} \,\in\, \begin{cases}
               {\rm O}(t_{d_j},\C)  & {\rm if\ } d_j \ {\rm even},\\
                {\rm Sp}(t_{d_j}/2, \C)  & {\rm if\ } d_j \ {\rm odd}\,.
             \end{cases}
\end{align}
Next  define $g\in {\rm Sp}(n,\C)$ as in  \eqref{real-nilpotent-sp-n-c}. Then the matrix $ [g]_\BC $ becomes a diagonal matrix. The first $t_{d_1}+ \cdots +t_{d_s}$ diagonal entries of $ [g]_\BC $ is of the form:
  \begin{align}\label{block-matrixDii}
 {\rm diag }\big( D_1,\ldots, D_s\big),\quad {\rm where \ } D_j=\begin{cases}
(-1)^{d_j-1} \sqrt{-1}\,{\rm I}_{t_{d_j}}&  {\rm if \ } d_j \ {\rm even},\\
(-1)^{d_j-1}{\rm I}_{t_{d_j}}&  {\rm if \ } d_j \ {\rm odd}\,.
    \end{cases}
\end{align}
As $ \nu$ is an involution, using \eqref{involution-tau} it follows that $A_{jj}^2\,= \,-{\rm I}_{t_{d_j}} $  when $d_j$ is even and $1\leq j \leq s$. Since  $\det A_{jj} = \pm1$, we conclude that $t_{d_j}$ is an even integer.	
\end{proof} 

We have a stronger result for nilpotent elements in $ {\rm Lie}({\rm PSp}(n,\C) )$.
\begin{theorem}\label{thm-nilpotent-psp-n-c}
	Every nilpotent  element in $ {\rm Lie}({\rm PSp}(n,\C) )$ is strongly real.
\end{theorem}
\begin{proof}
Let  $ X\in {\rm Lie}({\rm PSp}(n,\C) ) $ be a non-zero nilpotent element.
Let  $\{ X^l v^{d}_j \,\mid  \, 1\,\leq\, j \,\leq\, t_{d}, \  0 \,\leq\, l \,\leq\, d-1 , d\in \N_\d\}$ be a $\C$-basis of $\C^{2n}$ as in Proposition \ref{unitary-J-basis}. Recall that $ t_{d} $ is even for   $ d\in \O_\d $. 
 Define $ g\in {\rm GL}\,(\C^{2n} )$ as follows:
$$  
g(X^l v^{d}_j ) := \begin{cases}
(-1)^l \sqrt{-1} X^{l} v^{d}_j &  \qquad       \text{\ if \ } 1\leq j \leq t_d/2,\,  d\in \O_\d   \,,\\
(-1)^{l+1} \sqrt{-1} X^{l} v^{d}_j &  \qquad       \text{\ if \ } t_d/2<j\leq t_d,\, d\in \O_\d   \,,\\
(-1)^{l}\sqrt{-1} X^l v^{d}_{j} &\qquad \text{ if \ } d \in \E_\d\,  .
\end{cases}
$$ 
It follows that $ gX=-Xg $, $ g^2 = {\rm -Id} $. Using  Proposition \ref{unitary-J-basis}, we conclude that $\langle gx,gy \rangle = \langle x,y \rangle$ for all $ x,y\in \C^{2n} $. This shows that $ g\in {\rm Sp}(n,\C)$.     
\end{proof}

\section{Reality for nilpotent elements in classical Lie algebras over $\R$}\label{sec-reality-nilpotents--real-g}
In this section, we will consider the nilpotent elements in  simple Lie algebra over $\R $ of classical type. 

\subsection{The Lie algebras $\s\l_n(\R)$ \& $\s\l_n(\H)$}
In this subsection first we will consider nilpotent elements in $ \s\l_n(\R) $. Recall the definition of  $   \widetilde{\PC_{e}}(n) $ as in \eqref{def-Pe-n}.

\begin{theorem}\label{thm-sl-n-R-nilpot} 
Let $ X\in \s\l_n(\R) $ be a nilpotent element. Suppose that $X$ corresponds to the partition $\d \in \PC(n)$. Then  the following statements are equivalent :
\begin{enumerate}
\item $ X $ is  a real element.
\item $ X $ is a strongly  real element.
\item $\d\not\in   \widetilde{\PC_{e}}(n)$.
\end{enumerate}
\end{theorem}
\begin{proof}
Using Proposition \ref{J-basis}, $\{ X^l v^d_j \,\mid  \, 1\,\leq\, j \,\leq\, t_d, \  0 \,\leq\, l \,\leq\, d-1,\ d \,\in\, \N_\d \}$ is a $\R$-basis of $\R^n$. Define $g\in {\rm GL}_n(\R)$ as \eqref{defn-g-matrix}.
Then $g^2 = {\rm Id}$, $Xg=-gX$, and $\det g =\pm 1$. Next we will consider three separate cases:

{\it Case 1: } If $\d\in \PC_{\rm v.even}(n) $, then $\det g =1$. Hence, $X$ is strongly real.

{\it Case 2: } Next assume that $\d\in \PC(n)\setminus \PC_{\rm even}(n)$. Moreover, assume that $\det g=-1$. Suppose $d_{i_0}\in \O_\d$. i.e., $d_{i_0}$ is an odd integer.
Then define $\widetilde g \in {\rm GL}_n(\R)$ as follows:
$$
\widetilde  g(X^lv^d_j)\,=\,
\begin{cases}
(-1)^lX^lv^d_j & \text{ if either  } d\neq d_{i_0} \ {\rm or }\ \text{ when } d= d_{i_0},\ j>1  \\
 (-1)^{l+1} X^lv^d_j & \text{ if  }  d= d_{i_0},\ j=1\,.                            
\end{cases}
$$
Then $\det \widetilde g = -\det g =1$, $\widetilde g^2 = {\rm Id}$, $X\widetilde g=-\widetilde gX$.  Hence, in this case also $X$ is strongly real.

{\it Case 3 : } Finally assume that $\d\in \PC_{\rm even}(n)\setminus \PC_{\rm v.even}(n)$. In view of Lemma \ref{comm-Xd-even}, $X$ is real if and only if $\det g =1$. Therefore, in this case if $X$ is real, then it is strongly real. 
It follows from the definition of $g$ as in \eqref{defn-g-matrix} that $\det g =(-1)^{\sum_{\eta \in \E^2_\d} t_\eta }$.

Now $(1) \Rightarrow (3)$ follows from {\it Case 3}.  $(2)$ always implies $(1)$. Finally,  $(3) \Rightarrow (2)$ follows from {\it Case 1, Case 2 {\rm and} Case 3}.
This completes the proof.
\end{proof}

As a corollary we have the following known result:

\begin{corollary}[{\cite[Theorem 3.1.1]{ST}}]
Let $ X\in \s\l_n(\R) $ be a nilpotent element. Suppose that $ n\not\equiv 2 \pmod 4 $.  Then  the following statements are equivalent:
\begin{enumerate}
	\item $ X $ is  a real element.
	\item $ X $ is a strongly  real element.
\end{enumerate}
\end{corollary}
\begin{proof}
	Condition $ (3) $ of Theorem \ref{thm-sl-n-R-nilpot} always true for $ n\not\equiv 2 \pmod 4 $.
\end{proof}

Next we will consider the nilpotent elements in the Lie algebra $ \s\l_n(\H) $. As the Lie algebra $ \s\l_2(\H) $ is isomorphic to $\s\u(2)$ which is a compact Lie algebra, we will further assume that $n>2$.

\begin{theorem}\label{thm-sl-n-H}
	Every nilpotent element in $ \s\l_n(\H) $ is strongly real.
\end{theorem}
\begin{proof}
The proof follows from Proposition \ref{Propo-glnD}.
\end{proof}

\subsection{The Lie algebra $\s\u(p,q) $}\label{sec-su-pq}
As  we would like to address nilpotent elements; so it will be assumed that $p \,>\,0$ and $q\,>\,0$. Here $\<>$  denotes the Hermitian form on $\C^{p+q}$ defined by $\langle x,\, y \rangle \,:=\, \overline{x}^t{\rm I}_{p,q} y$, where ${\rm I}_{p,q}$ is as in \eqref{defn-I-pq-J-n}. We will  follow notation as defined in \S \ref{sec-notation}.

Let  $X\, \in\, {\s\u}(p,q)$ be a non zero nilpotent element, and let $\{X, H, Y\} \,\subset \,{\s\u}(p,q)$ be a $\s\l_2$-triple. 
We now apply Proposition \ref{unitary-J-basis}, Remark \ref{unitary-J-basis-rmk}(2).  Let $\{d_1,\, \ldots,\, d_s\}$, with $d_1 \,>\, \cdots \,>\, d_s$, be the finite ordered set  of integers that arise as $\R$-dimension of non-zero irreducible   $\text{Span}_\R \{ X,H,Y\}$-submodules of $\C^{p+q}$.
Let $t_{d_r} \,:=\, \dim_\C L(d_r-1)$ for $ 1 \,\leq \,r \,\leq\, s$. Then we have the corresponding partition 
${\d}\,:=\, [d_1^{t_{d_1} },\, \ldots ,\,d_s^{t_{d_s} }]  \,\in\, \PC(p+q)$. Recall that for the partition $ \d $,  $\E_{\d}:= \{d_i\mid d_i \text{ is even}\} $, and $\O_{\d}:= \{d_i\mid d_i \text{ is odd}\} $. 
 Let 
\begin{align}\label{def-p-d-su-pq}
p_d:=\begin{cases}
\#  \{ j \,\mid \, \langle v^d_j ,\, X^{d-1} v^d_j \rangle\,=\,  1\} &  {\rm if } \,
d\,\in\, \O_\d\,, \\
\#\{j \,\mid \,   \sqrt{-1}\langle v^d_j ,\, X^{d-1} v^d_j \rangle \,=\, 1\}  &{\rm if } \, d \,\in\,  \E_\d\,;
\end{cases} 
\quad{\rm and}\quad  q_d:= t_d- p_d \,.
\end{align} 
Recall that the signed Young diagram is uniquely determined by the sign of $ 1^{\rm st} $ column, see Definition \ref{def-young-diag} and Remark \ref{rmk-signed-young-diag-unique}. Now we associate a signed Young diagram for the element $ X $ by setting $ +1 $ in the  $ p_d $-many boxes and  $ -1 $ in the rest  $ q_d $-many boxes in the $ 1^{\rm st} $ column of the rectangular block of size $ t_d\times d $. It follows that  $ \d\in\PC(n) $ and $ (p_d, q_d) $ for $ d\in \N_\d $ are conjugation invariant and hence, they determined a well-defined signed Young diagram for the nilpotent orbit $ \OC_X .$ In fact,
{\it  the  nilpotent orbits in $ \s\u(p,q) $ are in bijection with  the signed Young diagram of signature $ (p,q) $}; see \cite[Theorem 9.3.3]{CoMc}. We refer to   \cite[Theorem 4.10]{BCM} for details regarding the parametring map.

\begin{example}\label{example-su-p-q}
We will consider two examples here. 
We will list down  signed Young diagrams associated to the non-zero nilpotent orbits in $ \s\u(2,2) $ and  $ \s\u(3,2)$, and then identify the real ones.
First we will consider the  signed Young diagrams for $ \s\u(2,2) $.

\ytableausetup{centertableaux, boxsize=1.2em}
(i) \blue{ \begin{ytableau}
		$\tiny{+1 }$ &	$\tiny{-1 }$ & $\tiny{+1 }$ & $\tiny{-1 }$  
\end{ytableau} }, \quad 
(ii)  \blue{ \begin{ytableau}
	$\tiny{+1 }$ & $\tiny{-1 }$ & $\tiny{-1 }$\\
	$\tiny{+1 }$
\end{ytableau} }\,, \  
(iii) \blue{  \begin{ytableau}
	$\tiny{-1 }$ & $\tiny{+1 }$ & $\tiny{+1 }$\\
	$\tiny{-1 }$
\end{ytableau}}, \quad 
(iv)   \begin{ytableau}
		 $\tiny{+1 }$ & $\tiny{-1 }$\\
		$\tiny{+1 }$\\ $\tiny{-1 }$
\end{ytableau} , 
(v)  \begin{ytableau}
		$\tiny{-1 }$ & $\tiny{+1 }$\\
		$\tiny{+1 }$\\ $\tiny{-1 }$
\end{ytableau} ,

(vi)  \begin{ytableau}
		$\tiny{+1 }$ & $\tiny{-1 }$\\
		$\tiny{+1 }$&$\tiny{-1 }$
\end{ytableau} , 
(vii) \blue{ \begin{ytableau}
		$\tiny{+1 }$  & $\tiny{-1 }$    \\
		$\tiny{-1 }$   &$\tiny{+1 }$
\end{ytableau}}, 
(viii) \begin{ytableau}
	$\tiny{-1 }$  & $\tiny{+1 }$    \\
	$\tiny{-1 }$   &$\tiny{+1 }$
\end{ytableau}\,. 

The nilpotent orbits corresponding to the blue colored diagrams are real as well as strongly real. Let  $\OC_X$ be an orbit whose signed Young diagram is given by either (iv) or  (vi). Then the relevant diagram of $ \OC_{-X} $ is given by  (v) or (viii), respectively.  
Next we will consider the  signed Young diagrams for $ \s\u(3,2) $.

\ytableausetup{centertableaux, boxsize=1.2em}
(i) \blue{ \begin{ytableau}
		$\tiny{+1 }$ &	$\tiny{-1 }$ & $\tiny{+1 }$ & $\tiny{-1 }$  & $\tiny{+1 }$
\end{ytableau} }, \quad 
(ii) \begin{ytableau}
	$\tiny{+1 }$ & $\tiny{-1 }$ &$\tiny{+1 }$  & $\tiny{-1 }$\\
	$\tiny{+1 }$
\end{ytableau}\,, \  
(iii) \begin{ytableau}
	$\tiny{-1 }$ & $\tiny{+1 }$ &$\tiny{-1 }$  & $\tiny{+1 }$\\
	$\tiny{+1 }$
\end{ytableau}, \quad 
(iv) \blue{  \begin{ytableau}
		$\tiny{-1 }$ & $\tiny{+1 }$ & $\tiny{+1 }$\\
		$\tiny{+1 }$ &$\tiny{-1 }$
\end{ytableau} }, \  
(v) \blue{  \begin{ytableau}
		$\tiny{-1 }$ & $\tiny{+1 }$ & $\tiny{+1 }$\\
		$\tiny{-1 }$ &$\tiny{+1 }$
\end{ytableau}}\,,

(vi) \blue{ \begin{ytableau}
		$\tiny{+1 }$  & $\tiny{-1 }$ &$\tiny{-1 }$    \\
		$\tiny{+1 }$\\
		$\tiny{+1 }$
\end{ytableau}}, 
(vii) \blue{ \begin{ytableau}
		$\tiny{-1 }$ & $\tiny{+1 }$ & $\tiny{+1 }$\\
		$\tiny{+1 }$ \\
		$\tiny{-1 }$
\end{ytableau}}, 
(viii)  \begin{ytableau}
	$\tiny{+1 }$  & $\tiny{-1 }$    \\
	$\tiny{+1 }$   &$\tiny{-1 }$\\
	$\tiny{+1 }$
\end{ytableau}, 
 (ix) \blue{ \begin{ytableau}
		$\tiny{+1 }$  & $\tiny{-1 }$    \\
		$\tiny{-1 }$   &$\tiny{+1 }$\\
		$\tiny{+1 }$
\end{ytableau}}, 
(x) \begin{ytableau}
	$\tiny{-1 }$  & $\tiny{+1 }$    \\
	$\tiny{-1 }$   &$\tiny{+1 }$\\
	$\tiny{+1 }$
\end{ytableau}, 
(xi)\begin{ytableau}
	$\tiny{+1 }$  & $\tiny{-1 }$    \\
	$\tiny{+1 }$  \\
	$\tiny{+1 }$\\
	$\tiny{-1 }$
\end{ytableau}, 
(xii)\begin{ytableau}
	$\tiny{-1 }$  & $\tiny{+1 }$    \\
	$\tiny{+1 }$  \\
	$\tiny{+1 }$\\
	$\tiny{-1 }$
\end{ytableau}.
The nilpotent orbits corresponding to the blue colored diagrams are real as well as strongly real. Let  $\OC_X$ be an orbit whose signed Young diagram is given by either (ii) or (viii) or (Xi). Then the relevant diagram of $ \OC_{-X} $ is given by  (iii) or (x) or (xii), respectively.
\qed\end{example}

\begin{theorem}\label{thm-supq-real-streal}
Let $X\in \s\u(p,q)$ be a non-zero nilpotent element. Let $ \d=\, [d_1^{t_{d_1} }, \ldots ,d_s^{t_{d_s} }] $ $\in \PC(p+q) $ be the partition  associated  to the orbit $\OC_X$ of $X$, and $ \E_{\d} $ be as in \eqref{Nd-Ed-Od}. Let  $p_{d_i}$ (resp. $q_{d_i}$) be the number of $+1$(resp. $-1$) occurred in the $1^{\rm st}$ column of the block of size $t_{d_i}{d_i}\times t_{d_i}{d_i}$ in the corresponding signed Young diagram. Then  the following statements are equivalent:	
\begin{enumerate}
		\item $X$ is real in $\s\u(p,q)$.
		\item $X$ is strongly real in $\s\u(p,q)$.
		\item  $p_\eta\,=\,q_\eta\, =\,t_\eta/2  $ for all $\eta\in \E_\d$.
\end{enumerate}	
\end{theorem}

\begin{proof}
We will show $ (3) \,\Rightarrow \,(2) \, \Rightarrow \, (1) \,\Rightarrow\,(3) $. 
Let $ \{X,H,Y\} \subset \s\u(p,q)$ be a $\s\l_2$-triple. Then $ \C^{p+q} $ has a $ \C $-basis of the form $ \{X^jv_j^d\mid d\in \N_\d \} $ which satisfy  Proposition \ref{unitary-J-basis} (3). 
 First assume that $(3)$ holds. i.e., $p_d=q_d$ for all $d\in \E_\d$.   After suitable reordering we may assume that 
$$
 \langle v_j^d, X^{d-1}v^d_j \rangle = \begin{cases}
 \sqrt{-1}   & \quad 1\leq j\leq t_d/2\,,\quad d\in \E_\d \\
 - \sqrt{-1} & \quad  t_d/2< j\leq t_d\,, \quad d\in \E_\d\\
 +1   & \quad 1\leq j\leq p_d\,,\quad d\in \O_\d \\
 {-1} & \quad  p_d< j\leq t_d\,, \quad d\in \O_\d\\
 \end{cases}
 $$
 Define $ g\in {\rm GL(\C^{p+q})} $ as follows:
$$  
g(X^l v^{d}_j ) := \begin{cases}
	(-1)^l X^{l} v^{d}_j &  \qquad \text{\ if }   d \,\in\O_\d^1\,  ,    
	\\
	(-1)^{l+1} X^{l} v^{d}_j &  \qquad  \text{\ if } d\in \O_\d^3 \,,  
	\\
	(-1)^{l} X^l v^{d_i}_{j+t_{d}/2} &\qquad \text{\ if } d \in \E_\d \,  , \ 1\leq j\leq t_{d}/2\,,\\
	(-1)^{l} X^l v^{d_i}_{j-t_{d}/2} &\qquad \text{\ if }  d\in \E_\d\,  ,  \  t_{d_i}/2< j\leq t_{d_i}\,. 
\end{cases}
$$
Then $ g\in {\rm SU}(p,q) $, $ g^2={\rm Id} $, and $ gX=-Xg $. This prove  $ (3)\Rightarrow (2) $.

Note that $(2)\Rightarrow (1) $ is always true. Thus, it remains to show $ (1)\Rightarrow (3)$.  Suppose $(1)$ holds.  $X$ is real in $\s\u(p,q)$, i.e., $-X\in \OC_X$, the nilpotent orbit of $X.$  
In view of Remark \ref{rmk-signed-young-diag-unique}, the signed Young diagram of $\OC_X$ is determined by $\d\in\PC(n)$ and $(p_d, q_d)$ for all $d\in \N_\d$ where $p_d$ and $q_d$ is given by \eqref{def-p-d-su-pq}. Note that $ \{-X,H,-Y\}$ is a $\s\l_2$-triple in $\s\u(p,q)$ containing $-X$ and the partition  $\d\in \PC(p+q)$ for the nilpotent element $-X$ is same as $X$, see Section \ref{sec-Jacobson-Morozov-partion}. Following the construction as done for $X$, let 
\begin{align}\label{def-p'-d-su-pq}
p'_d:=\begin{cases}
\#  \{ j \,\mid \, \langle v^d_j ,\, (-X)^{d-1} v^d_j \rangle\,=\,  1\} &  {\rm if } \,
d\,\in\, \O_\d\,, \\
\#\{j \,\mid \,   \sqrt{-1}\langle v^d_j ,\, (-X)^{d-1} v^d_j \rangle \,=\, 1\}  &{\rm if } \, d \,\in\,  \E_\d\,;
\end{cases} 
\quad{\rm and}\quad  q'_d := t_d- p'_d \,.
\end{align} 
Comparing \eqref{def-p-d-su-pq} and \eqref{def-p'-d-su-pq}, we conclude that 
$$
(p_d,q_d)= (p_d',q_d') \quad {\rm for }\,   d\in \O_{\d} \,, \quad   (p_d,q_d)= (q_d',p_d') \quad {\rm for }\,   d\in \E_{\d}  \,.
$$
As $-X\in \OC_X$,  the signed Young diagrams of $ \OC_X $ and $ \OC_{-X} $  coincide, see \cite[Theorem 9.3.3]{CoMc}. Hence,  $p_d=q_d$ for all $d\in \E_\d$. This completes the proof.
\end{proof}

\subsection{The Lie algebra $\s\o(p,q) $}\label{sec-so-pq}
Next we will consider the real simple  Lie algebra $ \s\o(p,q) $, see \S \ref{sec-Associated-Lie-gp} for the definition. Since we will deal with the nilpotent elements, we can  assume that $p,\, q>0$ . In this subsection, $\<>$ denotes the symmetric
form on  $\R^{p+q}$ defined by $\langle x,\, y \rangle \,:=\, x^t{\rm I}_{p,q} y$, where ${\rm  I}_{p,q}$ is as in \eqref{defn-I-pq-J-n}. Here we will consider the ${\rm  Ad}$-action of the group $ {\rm SO}(p,q) $ on the Lie algebra $ \s\o(p,q) $.

\begin{theorem}\label{propo-so-pq-unilpotent}
Every nilpotent element in $ \s\o(p,q) $ is strongly  {\rm Ad}$ _{{\rm SO}(p,q)}$-real. 
\end{theorem}
\begin{proof}
Let $X\in \s\o(p,q)$ be a non zero nilpotent element. Then $ \R^{p+q}$ has a basis of the form $ \{X^lv^j_d\,\mid \,0\leq l\leq d-1,\, 1\leq j\leq t_d,\, d\in \N_\d \} $ which satisfies Proposition \ref{unitary-J-basis}(3) and Remark \ref{unitary-J-basis-rmk}(1). Now define $ g\in {\rm GL}(\R^{p+q}) $ as follows :
\begin{align}\label{defn-g-so-pq-unilpotent}
g(X^l v^{d}_j ) := \begin{cases}
(-1)^l X^{l} v^ {d}_j &  \qquad \text{\ if  }   d\in \O_\d^1 \,, \\
(-1)^{l+1} X^{l} v^ {d}_j &  \qquad \text{\ if  } d\in \O_{\d}^3 \,, \\
(-1)^{l} X^l v^ {d}_{j+t_d/2} &\qquad \text{\ if  } d\in \E_{\d} \,  , \ 1\leq j\leq t_d/2\,,\\
(-1)^{l} X^l v^ {d}_{j-t_d/2} &\qquad \text{\ if  }  d\in \E_{\d} \,  ,  \  t_d/2< j \leq t_d \,. 
\end{cases}
\end{align}

Note that $\det g=1, \, gX=-Xg $ and $g\in {\rm SO}(p,q)$.
\end{proof}

\subsection{The Lie algebra $ \s\o^*(2n) $}\label{sec-so-*2n}
Here we will consider the real simple  Lie algebra $ \s\o^*(2n)  $. Recall that $ \H\,=\,\R + \,\ib\R+\,\jb\R +\,\kb\R $.  Throughout this subsection $\<>$ denotes the skew-Hermitian form on $\H^n$ defined by $\langle x, y \rangle \,:=\, \overline{x}^t \jb{\rm I}_{n} y$, for $x,y \in \H^n$. We will follow notation as defined in \S \ref{sec-notation}.

Let  $X\, \in\, {\s\o^*}(2n)$ be a non zero nilpotent element, and let $\{X, H, Y\} \,\subset \,{\s\o^*}(2n)$ be a $\s\l_2$-triple. 
We now apply Proposition \ref{unitary-J-basis}, Remark \ref{unitary-J-basis-rmk}(2). Following  Section \ref{sec-ordered-basis-V},
let $t_{d_r} \,:=\, \dim_\H L(d_r-1)$ for $ 1 \,\leq \,r \,\leq\, s$, and
${\d}\,:=\, [d_1^{t_{d_1} },\, \ldots ,\,d_s^{t_{d_s} }]  \,\in\, \PC(n)$ be the partition of $ n $ corresponding to the element $ X $.   Recall that for the partition $ \d $,  $\E_{\d}:= \{d_i\mid d_i \text{ is even}\} $, and $\O_{\d}:= \{d_i\mid d_i \text{ is odd}\} $; see \eqref{Nd-Ed-Od}. 
Let
\begin{align}\label{def-p-d-so*}
 p_d:=\begin{cases}
\#  \{ j \,\mid \, \langle v^d_j ,\, X^{d-1} v^d_j \rangle\,=\,  1\} &  {\rm when} \,
d\,\in\, \O_\d\,, \\
\#\{j \,\mid \,   \sqrt{-1}\langle v^d_j ,\, X^{d-1} v^d_j \rangle \,=\, 1\}  &{\rm when} \, d \,\in\,  \E_\d\,;
\end{cases} 
\quad{\rm and}\quad  q_d:= t_d- p_d \,.
\end{align}
Now we associate a signed Young diagram for the element $ X $ by setting $ +1 $ in the  $ p_\eta $-many boxes and  $ -1 $ in the rest  $ q_\eta $-many boxes in the $ 1^{\rm st} $ column of the rectangular block of size $ t_\eta\times \eta $ for $ \eta\in \E_\d $, and place $ +1 $ in the left most box of  odd length rows.  It follows that this association does not depend on the conjugacy class.  We refer to  \cite[Section 4.6]{BCM} for more details regarding the parametrizing map.   
{\it The nilpotent orbits in  $ \s\o^*(2n) $ is parametrized by the signed Young diagram of size $ n $ and any signature in which rows of odd length have their left most boxes labelled $ +1 $}, see \cite[Theorem 9.3.4]{CoMc}. 

\begin{theorem}\label{thm-so*2n-nilpot}
Let $ X $ be a nilpotent element in $ \s\o^*(2n) $.  Let $ \d=\, [d_1^{t_{d_1} }, \ldots ,d_s^{t_{d_s} }]  \in \PC(n) $ be the partition  corresponding to the orbit $\OC_X$ of $X$, and $ \E_{\d} $ be as in \eqref{Nd-Ed-Od}.  Then the following statements are equivalent :
\begin{enumerate}
\item $ X $ is real. 
\item  $ X $ is strongly real. 
\item $ p_\eta=q_\eta = t_\eta/2$ for all $ \eta\in \E_\d $.
\end{enumerate}
\end{theorem}

\begin{proof}
We will show $ (3) \,\Rightarrow \,(2) \, \Rightarrow \, (1) \,\Rightarrow\,(3) $.  Let $ \{X,H,Y\} \subset \g$ be a $\s\l_2$-triple corresponding to $ X $. Then $ \H^n $ has a $ \H $-basis of the form $ \{X^jv_j^d\mid d\in \N_\d \} $ which satisfy  Proposition \ref{unitary-J-basis}. 
First assume that $p_d=q_d$ for all $d\in \E_\d$. In view of Remark \ref{unitary-J-basis-rmk},  we can assume that 
$$
\langle v_j^d, X^{d-1}v^d_j \rangle = \begin{cases}
1   & \quad 1\leq j\leq p_d \,\quad d\in \E_\d \,,\\
-1  & \quad  p_d< j\leq t_d\, \quad d\in \E_\d\,,\\
{\bf j} &  \quad d\in \O_\d\,.\\
\end{cases}
$$
Define $ g\in {\rm GL}(\H^n) $ as follows:
$$  
g(X^l v^ {d}_j ) := \begin{cases}
(-1)^l X^{l} v^ {d}_j &  \text{if }  d\in \O_\d^  1, \\
(-1)^{l+1} X^{l} v^ {d}_j &  \text{if }  d\in\O_\d^3,    \\ 
(-1)^{l} X^l v^ {d}_{j+t_ {d}/2} & \text{if }  d\in \E_\d\,  , \ 1\leq j\leq t_ {d}/2\,,\\
(-1)^{l} X^l v^ {d}_{j-t_ {d}/2} & \text{if }  d\in \E_\d    \,  ,  \ { t_ {d}}/2< j\leq {t_ {d}}\,. 
\end{cases}
$$
Then $ g\in {\rm SO^*}(2n) $, $ g^2={\rm Id} $, and $ gX=-Xg $. This proves $ (2) $.

Note that $ (2)\Rightarrow (1) $ is always true. Thus, it remains to show $ (1)\Rightarrow (3) $. 
Suppose $(1)$ holds, i.e., $-X\in \OC_X$. Thus the partition  $\d\in \PC(n)$ for the nilpotent element $-X$ is same as $X$.
In view of Remark \ref{rmk-signed-young-diag-unique}, the signed Young diagram of $\OC_X$ is determined by $\d\in\PC(n)$ and $(p_d, q_d)$ for all $d\in \N_\d$ where $p_d$ and $q_d$ is given by \eqref{def-p-d-so*}. Note that $ \{-X,H,-Y\}$ is a $\s\l_2$-triple in $\s\o^*(2n)$ containing $-X$. 
Following the construction as done for $X$, let 
\begin{align}\label{def-p'-d-so*}
p'_d:=\begin{cases}
\#  \{ j \,\mid \, \langle v^d_j ,\, (-X)^{d-1} v^d_j \rangle\,=\,  1\} &  {\rm if } \,
d\,\in\, \O_\d\,, \\
\#\{j \,\mid \,   \sqrt{-1}\langle v^d_j ,\, (-X)^{d-1} v^d_j \rangle \,=\, 1\}  &{\rm if } \, d \,\in\,  \E_\d\,;
\end{cases} 
\quad{\rm and}\quad  q'_d := t_d- p'_d \,.
\end{align} 
Comparing \eqref{def-p-d-so*} and \eqref{def-p'-d-so*}, we conclude that 
$$
(p_d,q_d)= (p_d',q_d') \quad {\rm for }\,   d\in \O_{\d} \,, \quad   (p_d,q_d)= (q_d',p_d') \quad {\rm for }\,   d\in \E_{\d}  \,.
$$
Since, $-X\in \OC_X$,  the signed Young diagrams of $ \OC_X $ and $ \OC_{-X} $  coincide, see \cite[Theorem 9.3.4]{CoMc}. Hence,  $p_d=q_d$ for all $d\in \E_\d$. This completes the proof.
\end{proof}

\subsection{The Lie algebra $ \s\p(n,\R) $} \label{sec-sp-n-R}
Here we will consider the real simple  Lie algebra $ \s\p(n,\R) $.  Throughout this 
subsection $\<>$ denotes the symplectic form on $\R^{2n}$ defined by $\langle x, y \rangle := 
x^t{\rm J}_{n} y$, $x,\,y \,\in\, \R^{2n}$, where ${\rm J}_{n}$ is as in \eqref{defn-I-pq-J-n}.

Let  $X\, \in\, {\s\p}(n,\R)$ be a non zero nilpotent element. 
We now apply Proposition \ref{unitary-J-basis}, Remark \ref{unitary-J-basis-rmk}(2). 
Following  Section \ref{sec-ordered-basis-V},
let $t_{d_r} \,:=\, \dim_\R L(d_r-1)$ for $ 1 \,\leq \,r \,\leq\, s$, and
${\d}\,:=\, [d_1^{t_{d_1} },\, \ldots ,\,d_s^{t_{d_s} }]  \,\in\, \PC(2n)$ be the partition corresponding to the element $ X $.
Let
$$ p_\eta:=
\#  \{ j \,\mid \, \langle v^\eta_j ,\, X^{\eta-1} v^\eta_j \rangle\,=\,  1\} \,, 
\quad{\rm and}\quad  q_\eta:= t_\eta- p_\eta \quad {\rm for} \,
\eta\,\in\, \E_\d\,.
$$
Now we associate a signed Young diagram for the element $ X $ by setting $ +1 $ in the  $ p_\eta $-many boxes and  $ -1 $ in the rest  $ q_\eta $-many boxes in the $ 1^{\rm st} $ column of the rectangular block of size $ t_\eta\times \eta $ for $ \eta\in \E_\d $, and place $ +1 $ in the left most box of  odd length rows. It follows that this association does not depend on the conjugacy class.  We refer to  \cite[Section  4.7]{BCM} for details about  parametrizing map.
{\it The nilpotent orbits in  $ \s\p(n,\R) $ are parametrized by the signed Young diagram of size $2 n $ and any signature in which rows of odd length have their left most boxes labelled $ + 1$ and occur with even multiplicity}, see \cite[Theorem 9.3.5]{CoMc}.

\begin{theorem}\label{thm-sp-n-R-nilpot}
Let $ X $ be a nilpotent element in $ \s\p(n,\R) $. Let $ \d=\, [d_1^{t_{d_1} }, \ldots ,d_s^{t_{d_s} }]  \in \PC(2n) $ be the associated partition for  $X$, and $ \E_{\d} $ be as in \eqref{Nd-Ed-Od}. Then  the following statements are equivalent :	
\begin{enumerate}
\item $ X $ is real. 
\item$ X $ is strongly real.
\item   $ p_\eta=q_\eta = t_\eta/2$ for all $ \eta\in \E_\d $.
\end{enumerate}
\end{theorem}

\begin{proof}
We will show $ (3) \,\Rightarrow \,(2) \, \Rightarrow \, (1) \,\Rightarrow\,(3) $. 
Let $ \{X,H,Y\} \subset \s\p(n,\R) $ be a $\s\l_2$-triple corresponding to $ X $. 
Then $ \R^{2n} $ has a basis of the form $ \{X^jv_j^d\mid d\in \N_\d \} $ which satisfy  Proposition \ref{unitary-J-basis} (3).  In view of Remark \ref{unitary-J-basis-rmk}(1),  we can assume that 
$$
\langle v_j^d, X^{d-1}v^d_j \rangle = \begin{cases}
1   &  1\leq j\leq p_d \\
	-1  &   p_d< j\leq t_d
\end{cases}, d\in \E_\d \,;\  	\langle v_j^d, X^{d-1}v^d_{t_d/2+j} \rangle = 1 \,{\rm for \,}  d\in \O_\d ,1\leq j\leq t_d/2.
$$
First assume that $p_d=q_d$ for all $d\in \E_\d$.  Define $ g\in {\rm GL(\R^{2n})} $ as follows:
$$  
g(X^l v^ {d}_j ) := \begin{cases}
(-1)^l X^{l} v^ {d}_j &  \qquad \text{if }  d\in \O_\d^1   \,,\     
\\
(-1)^{l+1} X^{l} v^ {d}_j &  \qquad  \text{if }  d \in \O_\d^3 \,, 
\\
(-1)^{l} X^l v^ {d}_{j+t_d/2} &\qquad \text{if } d\in \E_{\d}\,  ,\ 1\leq j\leq t_d/2\,,\\
(-1)^{l} X^l v^ {d}_{j-t_d/2} &\qquad \text{if } d\in \E_{\d}\,  ,\ t_d/2< j\leq t_d\,. 
\end{cases}
$$
Then $ g\in {\rm Sp}(n,\,\R) $, $ g^2={\rm Id} $, and $ gX=-Xg $. This proves $ (2) $.

Note that $ (2)\Rightarrow (1) $ is always true. Thus, it remains to show $ (1)\Rightarrow (3) $.  
Let the signed Young diagram for the orbit $ \OC_{-X} $ be determined by $ \d'\in \PC(n) $ and $ (p_d',q_d') $ for $ d\in \N_\d $. Using Proposition \ref{unitary-J-basis} and Remark \ref{unitary-J-basis-rmk}(3), it follows that 
$$ \d=\d' \,;\quad (p_d,q_d)= (p_d',q_d') \quad {\rm for }\,   d\in \O_{\d} \,, \quad   (p_d,q_d)= (q_d',p_d') \quad {\rm for }\,   d\in \E_{\d}  \,.$$
Now $ (1) $ implies  $ -X\in \OC_X $. Therefore,  the signed Young diagrams of $ \OC_X $ and $ \OC_{-X} $  coincide, see \cite[Theorem 9.3.5]{CoMc}. Hence,  $p_d=q_d$ for all $d\in \E_\d$. This completes the proof.
\end{proof}

\subsection{The Lie algebra $\s\p(p,q)$} \label{sec-sp-pq}
Here we will deal with the nilpotent elements in the Lie algebra $\s\p(p,q)$. We will further assume that $p,\, q\, >\,0$.
Recall that $ \H\,=\,\R + \,\ib\R+\,\jb\R +\,\kb\R $.
 Throughout this subsection $\<>$ denotes the hermitian  form on  $\H^n$ defined by $\langle x,\, y \rangle \,:=\, \bar x^t{\rm I}_{p,q} y$, where
${\rm  I}_{p,q}$ is as in \eqref{defn-I-pq-J-n}.
Recall the definition of $\N_\d$ and $ \E_\d $ as in \eqref{Nd-Ed-Od}. 

\begin{theorem}\label{thm-st-real-psp-pq}
	Every non-zero nilpotent element in ${\rm Lie(PSp}(p,q))$ is real.
\end{theorem}
\begin{proof}
Let $ X\in \s\p(p,q) $ be a non zero nilpotent element. Then $\H^{p+q} $ has a $ \H $-basis of the form $ \{X^lv^j_d\mid d\in \N_\d \} $ which satisfies Proposition \ref{unitary-J-basis}(3). Now define $ g\in {\rm GL}(\H^{p+q}) $ as follows :
$$  
g(X^l v^{d}_j ) :=(-1)^{l}  X^l v^{d}_{j}\ib  \,,\quad  d\in \N_\d  \, .
$$
Then $  g^2={\rm -Id}, \, gX=-Xg $ and $ g\in {\rm Sp}(p,q) $.
This completes the proof.
\end{proof}	

Now we have an immediate corollary  which extends \cite[Theorem 1.1]{BG}  for any unipotent element in the higher rank situation.
\begin{corollary}\label{cor-sp-pq-real}
		Every non-zero nilpotent element in $\s\p(p,q)$ is real.
\end{corollary}

Now we associate a signed Young diagram for the element $ X $ by setting $ +1 $ in the  $ p_\eta $-many boxes and  $ -1 $ in the rest  $ q_\eta $-many boxes in the $ 1^{\rm st} $ column of the rectangular block of size $ t_\eta\times \eta $ for $ \eta\in \E_\d $, and place $ +1 $ in the left most box of  even length rows.  It follows that this association does not depend on the conjugacy class.  We refer to  \cite[Section 4.7]{BCM} for more details regarding the parametrizing map.   
{\it The nilpotent orbits in  $ \s\p(p,q)$ are parametrized by the signed Young diagram of signature $ (p,q) $  in which rows of  even length have their left most boxes labelled $ +1 $}, see \cite[Theorem 9.3.5]{CoMc}.

\begin{theorem}\label{thm-sp-pq-st-real}
Let $ X\in \s\p(p,q) $ be a nilpotent.  Let $ \d=\, [d_1^{t_{d_1} }, \ldots ,d_s^{t_{d_s} }]  \in \PC(p+q) $ be the partition  corresponding to the orbit $\OC_X$ of $X$, and $ \E_{\d} $ be as in \eqref{Nd-Ed-Od}.   Then $ X $ is a strongly real element if and only if  $ t_\eta $ is even for all $ \eta\in \E_\d $.
\end{theorem}

\begin{proof}
Let $ X\in \s\p(p,q) $ be a non zero nilpotent element. Then $\H^{p+q} $ has a $ \H $-basis of the form $ \{X^lv^j_d\mid d\in \N_\d \} $ which satisfies Proposition \ref{unitary-J-basis}(3).  Suppose that $ t_\eta $ is even for all $ \eta\in \E_\d $.  Define $ g\in {\rm GL}(\H^{p+q}) $ as follows :
$$  
g(X^l v^{d}_j ) := \begin{cases}
(-1)^l X^{l} v^{d}_j & \text{ if }  \ d\in \O_{\d}^1    \,,  \\
(-1)^{l+1} X^{l} v^{d}_j & \text{ if } d \in \O_{\d}^3\,, \\
(-1)^{l} ( X^l v^{d}_{j+t_d/2} )\ib & \text{ if }   d\in \E_\d \, , \ 1\leq j\leq t_d/2\,,
\\
(-1)^{l+1}  (X^l v^ {d}_{j-t_d/2})\ib & \text{ if }  d\in \E_\d \,  ,  \  t_d/2< j\leq t_d\,. 
\end{cases}
$$
Then $ g^2={\rm Id}, \, gX=-Xg $ and $ g\in {\rm Sp}(p,q)$. Hence, $X$ is strongly real.

Next assume that $X$ is a strongly real nilpotent element in  $\s\p(p,q)$. Let  $\tau \in Z_{{\rm Sp}(p,q)}  (X)$.   Recall that $ \BC $ is the ordered basis of  $ \H^{p+q}$ as in \eqref{old-ordered-basis}. Following \S \ref{sec-ordered-basis-V}, let $ [\tau]_\BC= (A_{ij}) $, and 
$\tau_D$ be the block diagonal part of  $ [\tau]_\BC$, i.e., consists of   only the matrices $(A_{jj})$ in the diagonal. Then $ \tau_D X=X \tau_D$ and $ \tau_D H=H \tau_D$. Using Lemma \ref{comm-XH}, it follows that $ \tau_D\in Z_{{\rm Sp}(p,q)}  (X,H,Y)$.  Using Proposition \ref{cor-centralizer-sl2-triple-sp-pq}, we conclude  that for $1\leq j\leq s$
\begin{align}\label{block-matrixAii-sp-pq}
  A_{jj} \,\in\, {\rm SO}^*(2t_{d_j})  \quad {\rm if\ } d_j \ {\rm even}.
\end{align}
Next  define $g\in {\rm Sp}(p,q)$ as in the proof of Theorem \ref{thm-st-real-psp-pq}. Then the matrix $ [g]_\BC $ becomes a diagonal matrix. The first $t_{d_1}+ \cdots +t_{d_s}$ diagonal entries of $ [g]_\BC $ is of the form:
\begin{align*}
{\rm diag }\big( D_1,\ldots, D_s\big),\quad {\rm where \ } D_j\,=\, (-1)^{d_j-1} \ib \,{\rm I}_{t_{d_j}} \quad {\rm for } \,1\leq j\leq s \,.
\end{align*}
Since $ X $ is strongly real,  $\nu X \nu^{-1} =-X$ for some involution $ \nu\in {\rm Sp}(p,q)$.  Then   $  \nu\, \in\, g Z_G(X) $.
It follows that  when $d_j$ is even and $1\leq j \leq s$, the matrices $A_{jj}$ also satisfy
\begin{align}\label{rel-Ajj-matrix}
(A_{jj}\ib )^2\,= \,{\rm I}_{t_{d_j}} \,.
\end{align}
Using \eqref{block-matrixAii-sp-pq} and \eqref{rel-Ajj-matrix} we conclude that 
$$
A_{jj} =-\bar A_{jj}^t $$
Let $\lambda$ be a right eigen value of $A_{jj}$. Then it follows that $-\bar \lambda^{-1}$ is also a right eigen value which is distinct from $\lambda$. This proves that $t_{d_j}$, the order of the matrix $A_{jj}$,
 must be even. This completes the proof.
\end{proof}

\subsection*{Acknowledgements} We thank P. Chatterjee, D. Prasad, A. Singh and M. Thakur for their comments on a first draft of this article. We also thank the anonymous referee for many comments and suggestion.  

Gongopadhyay acknowledges partial support by the SERB core research grant   CRG/2022/003680. 
Maity is supported by an NBHM PDF during the course of this work.

\end{document}